\pdfoutput=1
\documentclass[11pt,oneside]{amsart}
\usepackage[paperheight=279mm,paperwidth=18cm,textheight=26cm,textwidth=14cm]{geometry}
\usepackage{amsfonts,amssymb,amsmath,amsthm}
\usepackage[T1]{fontenc}
\usepackage[utf8]{inputenc}
\usepackage{mathtools}
\usepackage[unicode,hidelinks,bookmarks]{hyperref}
\hypersetup{
colorlinks=true,
citecolor=[rgb]{0,0,0.5},
linkcolor=[rgb]{0,0,0.5},
urlcolor=[rgb]{0,0,0.75},
pdfpagemode=UseNone,
pdfstartview=FitH,
pdfdisplaydoctitle=true,
pdftitle={Bootstrapping jump inequalities},
pdfauthor={Mirek, Stein, Zorin-Kranich},
pdflang=en-US
}

\usepackage[style=alphabetic,sorting=nyt]{biblatex}
\numberwithin{equation}{section}
\newtheorem{theorem}[equation]{Theorem}

\newtheorem{lemma}[equation]{Lemma}
\newtheorem{proposition}[equation]{Proposition}
\theoremstyle{definition}

\newcommand{\bbU}{\mathbb{U}}
\newcommand{\V}{\mathbb{V}}
\newcommand{\D}{\mathbb{D}}
\newcommand{\J}{\mathbb{J}}
\newcommand{\R}{\mathbb{R}}
\newcommand{\N}{\mathbb{N}}
\newcommand{\Z}{\mathbb{Z}}

\newcommand{\T}{\mathbb{T}}

\newcommand{\GG}{\mathbb{G}}
\newcommand{\dif}{\mathrm{d}}

\renewcommand{\C}{\mathbb{C}}
\newcommand{\bfa}{\mathbf{a}}
\newcommand{\bfA}{\mathbf{A}}
\newcommand{\calA}{\mathcal{A}}

\newcommand{\calP}{\mathcal{P}}

\newcommand{\calB}{\mathcal{B}}
\newcommand{\calM}{\mathcal{M}}

\newcommand{\calH}{\mathcal{H}}

\newcommand{\calS}{\mathcal{S}}

\newcommand{\frakg}{\mathfrak{g}}
\newcommand{\frakm}{\mathfrak{m}}
\newcommand{\frakq}{\mathfrak{q}}
\newcommand{\bbI}{\mathbb{I}}
\newcommand{\FT}{\mathcal{F}} 

\newcommand{\loc}{\mathrm{loc}}

\newcommand{\Dini}{\mathrm{Dini}}
\newcommand{\logDini}{\mathrm{logDini}}

\newcommand*{\DMO}[1]{\expandafter\DeclareMathOperator\csname #1\endcsname {#1}}
\DMO{supp}
\DMO{cl}
\DMO{lcm}
\DMO{dist}
\DMO{tr}
\DMO{diam}

\DeclarePairedDelimiter\abs{\lvert}{\rvert}
\DeclarePairedDelimiter\meas{\lvert}{\rvert}
\DeclarePairedDelimiter\norm{\lVert}{\rVert}
\DeclarePairedDelimiter\card{\lvert}{\rvert}

\DeclarePairedDelimiterX\innerp[2]{\langle}{\rangle}{#1,#2}
\providecommand\given{}
\newcommand\SetSymbol[1][]{%
\nonscript\:#1\vert
\allowbreak
\nonscript\:
\mathopen{}}
\DeclarePairedDelimiterX\Set[1]\{\}{\renewcommand\given{\SetSymbol[\delimsize]}#1}

\newcommand{\one}{\mathbf{1}}
\newcommand{\ind}[1]{\one_{#1}}

\begin{document}
\title[Jump inequalities]{A bootstrapping approach to jump inequalities\\ and their applications}
\author{Mariusz Mirek}
\address[Mariusz Mirek]{
Department of Mathematics,
Rutgers University,
Piscataway, NJ 08854, USA \&
Instytut Matematyczny,
Uniwersytet Wroc{\l}awski,
Plac Grunwaldzki 2/4,
50-384 Wroc{\l}aw
Poland}
\email{mirek@math.uni.wroc.pl}

\author{Elias M.\ Stein}
\address[Elias M.\ Stein]{
Department of Mathematics,
Princeton University,
Princeton,
NJ 08544-100 USA}
\email{stein@math.princeton.edu}

\author{Pavel Zorin-Kranich}
\address[Pavel Zorin-Kranich]{
Mathematical Institute,
University of Bonn,
Endenicher Allee 60,
53115 Bonn,
Germany}
\email{pzorin@math.uni-bonn.de}

\begin{abstract}
The aim of this paper is to present an abstract and general approach to jump inequalities in harmonic analysis. Our principal conclusion is the refinement of $r$-variational estimates, previously known for $r>2$, to end-point results for the jump quasi-seminorm corresponding to $r=2$.
This is applied to the dimension-free results recently obtained by the first two authors in collaboration with Bourgain, and Wr\'obel, and also to operators of Radon type treated by Jones, Seeger, and Wright.
\end{abstract}

\thanks{Mariusz Mirek was partially supported by the Schmidt Fellowship and the IAS Found for Math.\ and by the National Science Center, NCN grant DEC-2015/19/B/ST1/01149.
Elias M.\ Stein was partially supported by NSF grant DMS-1265524.
Pavel Zorin-Kranich was partially supported by the Hausdorff Center for Mathematics and DFG SFB-1060.}

\maketitle

\section{Introduction}
\label{sec:introduction}
Variational and jump inequalities in harmonic analysis, probability, and ergodic theory have been studied extensively since \cite{MR1019960}, where a variational version of the Hardy--Littlewood maximal function was introduced. The purpose of this paper is to formulate general sufficient conditions that allow us to deal with variational and jump inequalities for a wide class of operators. Our approach will be based on certain bootstrap arguments. As an application we extend the known $L^p$ estimates for $r$-variations for $r>2$ (see definition \eqref{eq:def:Vr}) to end-point assertions for the jump quasi-seminorm $J_2^p$ (see definition \eqref{eq:jump-space}), which corresponds to $r=2$. In this way our results will extend previously recently obtained assertions in \cite{MR3777413} and \cite{arxiv:1804.07679} for dimension-free estimates given for $r>2$, as well as a number of results in~\cite{MR2434308} for operators of Radon type.

We recall the notation for jump quasi-seminorms from \cite{arxiv:1808.04592}.
For any $\lambda>0$ and $\bbI \subset \R$ the \emph{$\lambda$-jump counting function} of a function $f : \bbI \to \C$ is defined by
\begin{align}
\label{eq:def:jump}
\begin{split}
N_{\lambda}(f)& :=
N_{\lambda}(f(t) : t\in\bbI)\\
&:=\sup \Set{J\in\N \given \exists_{\substack{t_{0}<\dotsb<t_{J}\\ t_{j}\in\bbI}} : \min_{0<j\leq J} \abs{f(t_{j})-f(t_{j-1})} \geq \lambda}.
\end{split}
\end{align}
and the $r$-variation seminorm by
\begin{align}
\label{eq:def:Vr}
\begin{split}
V^{r}(f)
:=&V^{r}(f(t) : t\in\bbI)\\
:=&
\begin{cases}
\sup_{J\in\N} \sup_{\substack{t_{0}<\dotsb<t_{J}\\ t_{j}\in\bbI}}\Big(\sum_{j=1}^{J} \abs{f(t_{j})-f(t_{j-1})}^{r} \Big)^{1/r},
&
0< r <\infty,\\
\sup_{\substack{t_{0}<t_{1}\\ t_{j}\in\bbI}} \abs{f(t_{1})-f(t_{0})},
& r = \infty,
\end{cases}
\end{split}
\end{align}
where the former supremum is taken over all finite increasing sequences in $\bbI$.

Throughout the article $(X,\calB,\frakm)$ denotes a $\sigma$-finite measure space. For a function $f : X \times \bbI \to \C$ the jump quasi-seminorm on $L^p(X)$ for $1<p<\infty$ is defined by
\begin{align}
\label{eq:jump-space}
\begin{split}
J^{p}_{2}(f):=J^{p}_{2}(f : X \times \bbI \to \C)&:=J^{p}_{2}((f(\cdot, t))_{t\in\bbI}):=J^{p}_{2}((f(\cdot, t))_{t\in\bbI} : X \to \C) \\
&:=\sup_{\lambda>0} \norm[\big]{ \lambda N_{\lambda}(f(\cdot, t): t\in \bbI)^{1/2} }_{L^{p}}.
\end{split}
\end{align}

In this connection by \cite[Lemma 2.12]{arxiv:1808.04592} we note that
\begin{equation}
\label{eq:V-vs-jump}
\norm{V^{r}(f)}_{L^{p,\infty}}
\lesssim_{p, r}
J^{p}_{2}(f)
\leq
\norm{V^{2}(f)}_{L^p}
\end{equation}
for $r>2$, and the first inequality fails for $r=2$.

We now briefly list our main results.

\begin{enumerate}
\item The extension to the jump quasi-seminorm $J^p_2$ of dimension-free estimates for maximal averages over convex sets, as given by Theorem~\ref{thm:body-dyadic}, Theorem~\ref{thm:body-short} and Theorem~\ref{thm:lq-ball-short} below.
\item The corresponding extension to $J^p_2$ of the previous dimension-free estimates for cubes in the discrete setting, see Theorem~\ref{jump-discrete}.
\item The general $J^p_2$ results for operators of Radon type (both averages and singular integrals) in Theorem~\ref{thm:radon-av} and Theorem~\ref{thm:radon-sing}, related to the previous results in \cite{MR2434308}.
\end{enumerate}

Underlying the proofs of all these results will be the basic facts about the jump quantity $J^p_2$ obtained in our recent paper \cite{arxiv:1808.04592}, and the bootstrap arguments in Section~\ref{sec:gener-prel} of the present paper.
The reader might compare the methods in Section~\ref{sec:gener-prel} with related arguments in \cite[Section 2.2]{MR3777413} as well as \cite{MR0466470}, \cite{MR837527}, \cite{MR828824}, and Christ's observation included in \cite{MR949003}.
The techniques in Section~\ref{sec:gener-prel} will be carried out in the following framework. We assume that we are given a measure space $(X,\calB,\frakm)$ which is endowed with a sequence of linear operators $(S_j)_{j\in\Z}$ acting on $L^1(X)+ L^{\infty}(X)$ that play the role of the Littlewood--Paley operators. Namely, the following conditions are satisfied:
\begin{enumerate}
\item The family $(S_j)_{j\in\Z}$ is a resolution of the identity on $L^2(X)$, i.e. the identity
\begin{align}
\label{eq:3}
\sum_{j\in\Z}S_j = \operatorname{Id}
\end{align}
holds in the strong operator topology on $L^2(X)$.
\item For every $1<p<\infty$ we have
\begin{align}
\label{eq:4}
\norm[\Big]{\big(\sum_{j\in\Z} \abs{S_{j} f}^{2} \big)^{1/2}}_{L^p}
\lesssim
\norm{f}_{L^p}, \qquad f\in L^p(X).
\end{align}
\end{enumerate}

Suppose now we have a family of linear operators $(T_t)_{t\in\bbI}$ acting on $L^{1}(X) + L^{\infty}(X)$, where the index set $\bbI$ is a countable subset of $(0, \infty)$. We assume that $\bbI\subseteq (0, \infty)$ to make our exposition consistent with the results in the literature. One of our aims is to understand what kind of conditions have to be imposed on the family $(T_t)_{t\in\bbI}$, in terms of its interactions with the Littlewood--Paley operators $(S_j)_{j\in\Z}$ to obtain the inequality
\begin{align}
\label{eq:18}
J^{p}_{2}((T_tf)_{t\in\bbI} : X \to \C)\lesssim\norm{f}_{L^p}
\end{align}
in some range of $p$'s. We accomplish this task in Section~\ref{sec:gener-prel} by proving Theorem~\ref{thm:duo+rubio:max} and Theorem~\ref{thm:short-var} for positive operators\footnote{A linear operator $T$ is positive if $Tf\ge 0$ for every $f\ge0$.} by certain bootstrap arguments, and Theorem~\ref{thm:duo+rubio:jump} for general operators. Our approach will be based on extension of ideas from \cite{MR837527} and \cite{MR3777413} to a more abstract setting.

As mentioned above it has been very well known since Bourgain's article \cite{MR1019960} that $r$-variational estimates (and consequently maximal estimates, see \eqref{eq:def:Vr}) can be deduced from jump inequalities.
Namely, a priori jump estimates \eqref{eq:18} in an open range of $p\in (1, \infty)$ imply
\[
\norm{V^{r}(T_tf : t\in\bbI)}_{L^p}\lesssim_{p, r}\norm{f}_{L^p}
\]
in the same range of $p$'s and for all $r\in(2, \infty]$. This follows from \eqref{eq:V-vs-jump} and interpolation.
Therefore, it is natural to say that the jump inequality in \eqref{eq:8} is an endpoint for $r$-variations at $r=2$. On the other hand, we also know that the range of $r\in(2, \infty]$ in $r$-variational estimates, for many operators in harmonic analysis, is sharp due to the sharp estimates in L\'epingle's inequality for martingales, see \cite{arxiv:1808.04592} and the references therein.

Here and later we write $a \lesssim b$ if $a \leq C b$, where the constant $0<C<\infty$ is allowed to depend on $p$, but not on the underlying abstract measure space $X$ or function $f$. If $C$ is allowed to depend on some additional parameters this will be indicated by adding a subscript to the symbol $\lesssim$.

\subsection{Applications to dimension-free estimates}

An important application of the results from Section~\ref{sec:gener-prel} will be bounds independent of the dimension in jump inequalities associated with the Hardy--Littlewood averaging operators. Let $G \subset \R^{d}$ be a symmetric convex body, that is, a non-empty symmetric convex open bounded subset of $\R^d$.
Define for $t>0$ and $x\in\R^d$ the averaging operator

\begin{equation}
\label{eq:av-op}
\calA_{t}^{G}f(x) := \abs{G}^{-1} \int_{G} f(x-ty) \dif y, \quad f\in L^1_{\loc}(\R^d).
\end{equation}

It follows from the spherical maximal theorem that in the case that $G$ is the Euclidean ball the maximal operator $\calA_{\star}^Gf:=\sup_{t>0} \abs{\calA_{t}^{G}f}$ corresponding to \eqref{eq:av-op} is bounded on $L^{p}(\R^{d})$ for all $p>1$, uniformly in $d\in\N$ \cite{MR663787}.
This result was extended to arbitrary symmetric convex bodies $G\subset\R^d$ in \cite{MR868898} (for $p=2$) and \cite{MR853451,MR828824} (for $p>3/2$).
For unit balls $G=B^q$ induced by $\ell^{q}$ norms in $\R^d$ the full range $p>1$ of dimension-free estimates was established in \cite{MR1042048} (for $1\leq q < \infty$) and \cite{MR3273441} (for cubes $q=\infty$) with constants depending on $q$.
In the latter case the product structure of the cubes is important; this result was recently extended to products of Euclidean balls of arbitrary dimensions \cite{arxiv:1703.07728}.

Variational versions of most of the aforementioned dimension-free estimates were obtained in \cite{MR3777413} for $r>2$.
In this article we give a shorter and more self-contained proof of the main results of \cite{MR3777413} and extend them to the endpoint $r=2$ by appealing to Theorem~\ref{thm:duo+rubio:max} and Theorem~\ref{thm:short-var}.
A notable simplification is that we do not use the maximal estimates as a black box.
In particular, we reprove all dimension-free estimates for the maximal function $\calA_{\star}^{G}$.

In view of \eqref{eq:V-vs-jump} and by real interpolation, Theorem~\ref{thm:body-dyadic} below extends \cite[Theorem 1.2]{MR3777413}.

\begin{theorem}
\label{thm:body-dyadic}
Let $d\in\N$ and $G \subset \R^{d}$ be a symmetric convex body.
Then for every $1<p<\infty$ and $f\in L^{p}(\R^{d})$ we have
\begin{equation}
\label{eq:body-dyadic}
J^{p}_{2}( (\calA_{2^k}^{G} f)_{k\in\Z} : \R^{d} \to \C)
\lesssim
\norm{f}_{L^{p}},
\end{equation}
where the implicit constant is independent of $d$ and $G$.
\end{theorem}

As a consequence of Theorem~\ref{thm:body-dyadic} and the decomposition into long and short jumps, see \eqref{eq:8}, Theorems~\ref{thm:body-short} and~\ref{thm:lq-ball-short} below extend \cite[Theorem 1.1]{MR3777413} and \cite[Theorem 1.3]{MR3777413}, respectively.
Hence Theorem~\ref{thm:body-dyadic}
can be thought of as the main result of this paper, since inequalities
\eqref{eq:body-short} and \eqref{eq:body-short-lq} were obtained in
\cite{MR3777413}. However, we shall present a different approach
to establish the estimates in \eqref{eq:body-short} and \eqref{eq:body-short-lq}.

\begin{theorem}
\label{thm:body-short}
Let $G$ be as in Theorem~\ref{thm:body-dyadic}.
Then for every $3/2 < p < 4$ and $f\in L^{p}(\R^{d})$ we have
\begin{equation}
\label{eq:body-short}
\norm[\Big]{\Big( \sum_{k\in\Z} \big(V^{2}(\calA_{t}^{G}f : t \in [2^k,2^{k+1}])\big)^{2} \Big)^{1/2}}_{L^{p}}
\lesssim
\norm{f}_{L^{p}}.
\end{equation}
In particular,
\begin{equation}
\label{eq:body-short:cor}
J^{p}_{2}((\calA_{t}^{G} f)_{t>0} : \R^{d} \to \C)
\lesssim
\norm{f}_{L^{p}},
\end{equation}
where the implicit constants in \eqref{eq:body-short} and \eqref{eq:body-short:cor} are independent of $d$ and $G$.
\end{theorem}
\begin{theorem}
\label{thm:lq-ball-short}
Let $d \in \N$ and $G \subset \R^{d}$ be the unit ball induced by the $\ell^{q}$ norm in $\R^d$ for some $1\le q\le \infty$. Then for every $1 < p < \infty$ and $f\in L^{p}(\R^{d})$ we have
\begin{equation}
\label{eq:body-short-lq}
\norm[\Big]{\Big( \sum_{k\in\Z} \big(V^{2}(\calA_{t}^{G}f : t \in [2^k,2^{k+1}])\big)^{2} \Big)^{1/2}}_{L^{p}}
\lesssim_{q}
\norm{f}_{L^{p}}.
\end{equation}
In particular
\begin{equation}
\label{eq:body-short:cor-lq}
J^{p}_{2}((\calA_{t}^{G} f)_{t>0} : \R^{d} \to \C)
\lesssim_{q}
\norm{f}_{L^{p}},
\end{equation}
where the implicit constants in \eqref{eq:body-short-lq} and \eqref{eq:body-short:cor-lq} are independent of $d$.
\end{theorem}

The method of the present paper also allows us to provide estimates independent of the dimension in jump inequalities associated with the discrete averaging operator along cubes in $\Z^d$. For every $x\in\Z^d$ and $N\in\N$ let
\begin{align}
\label{eq:37}
\bfA_Nf(x) := \frac{1}{\card{Q_N\cap \Z^d}}\sum_{y\in Q_N\cap\Z^d}f(x-y), \qquad f\in\ell^1(\Z^d),
\end{align}
be the discrete Hardy--Littlewood averaging operator, where $Q_N=[-N, N]^d$.
\begin{theorem}
\label{jump-discrete}
For every $3/2 < p < 4$ and $f\in \ell^p(\Z^d)$ we have
\begin{align}
\label{eq:38}
J^{p}_{2}((\bfA_{N} f)_{N\in\N} : \Z^{d} \to \C)
\lesssim
\norm{f}_{\ell^p}.
\end{align}
Moreover, if we consider only lacunary parameters, then \eqref{eq:38} remains true for all $1<p<\infty$ and we have
\begin{equation}
\label{eq:39}
J^{p}_{2}((\bfA_{2^k} f)_{k\ge0} : \Z^{d} \to \C)
\lesssim
\norm{f}_{\ell^p},
\end{equation}
where the implicit constants in \eqref{eq:38} and \eqref{eq:39} are independent of $d$.
\end{theorem}
Theorem~\ref{jump-discrete} provides the endpoint estimate at $r=2$ for the recent dimension-free estimates \cite{arxiv:1804.07679} for $r$-variations corresponding to operator \eqref{eq:37}.

The dimension-free results are proved in Section~\ref{sec:convex} by combining the results from Section~\ref{sec:gener-prel} (Theorem~\ref{thm:duo+rubio:max} and Theorem~\ref{thm:short-var}) with the jump estimates for the Poisson semigroup from \cite{arxiv:1808.04592} and Fourier multiplier estimates from \cite{MR868898} and \cite{MR1042048,MR3273441}.

\subsection{Applications to operators of Radon type}
Another important class of operators which was extensively studied in~\cite{MR2434308} in the context of jump inequalities are operators of Radon type modeled on polynomial mappings.

Let $P=(P_1,\dotsc, P_d):\R^k\to\R^d$ be a polynomial mapping, where each component $P_j:\R^k\to\R$ is a polynomial with $k$ variables and
real coefficients.
We fix $\Omega\subset\R^k$ a convex open bounded set containing the origin (not necessarily symmetric), and for every $x\in\R^d$ and $t>0$ we define the Radon averaging operators
\begin{align}
\label{eq:40}
\calM_{t}^P f(x):= \frac{1}{\meas{\Omega_{t}}} \int_{\Omega_{t}} f(x-P(y)) \dif y,
\end{align}
where $\Omega_t=\Set{x\in\R^k \given t^{-1}x\in\Omega}$. Using Theorem~\ref{thm:duo+rubio:max} and Theorem~\ref{thm:short-var} we easily deduce Theorem~\ref{thm:radon-av}, see Section~\ref{sec:radon}.
\begin{theorem}
\label{thm:radon-av}
For every $1<p<\infty$ and $f\in L^{p}(\R^{d})$ we have
\begin{equation}
\label{eq:41}
J^{p}_{2}( (\calM_{t}^{P} f)_{t>0} : \R^{d} \to \C)
\lesssim_{d, p}
\norm{f}_{L^{p}},
\end{equation}
where the implicit constant is independent of the coefficients of $P$.
\end{theorem}

Before we formulate a corresponding result for truncated singular integrals we need to fix some definitions and notation.
A \emph{modulus of continuity} is a function $\omega :[0,\infty) \to [0,\infty)$ with $\omega(0)=0$ that is subadditive in the sense that
\[
u\leq t+s
\implies
\omega(u)\leq \omega(t)+\omega(s).
\]
Substituting $s=0$ one sees that $\omega(u)\leq\omega(t)$ for all $0\leq u\leq t$.
The basic example is $\omega(t)=t^{\theta}$, with $\theta\in(0, 1)$.
Note that the composition and sum of two moduli of continuity is again a modulus of continuity.
In particular, if $\omega(t)$ is a modulus of continuity and $\theta\in(0, 1)$, then $\omega(t)^{\theta}$ and $\omega(t^{\theta})$ are also moduli of continuity.

The \emph{Dini norm} and the \emph{log-Dini norm} of a modulus of
continuity are defined respectively by setting
\begin{equation}
\label{eq:Dini}
\norm{\omega}_{\Dini} := \int_{0}^{1} \omega(t) \frac{\dif t}{t},
\quad \text{ and } \quad
\norm{\omega}_{\logDini} := \int_{0}^{1} \omega(t) \frac{\abs{\log t}\dif t}{t}.
\end{equation}
For any $c>0$ the integral can be equivalently (up to a $c$-dependent multiplicative constant) replaced by the sum over $2^{-j/c}$ with $j\in\N$.

Finally, for every $x\in\R^d$ and $t>0$ we will consider the truncated singular Radon transform
\begin{align}
\label{eq:47}
\calH_{t}^P f(x) := \int_{\R^{k} \setminus \Omega_{t}} f(x-P(y)) K(y) \dif y,
\end{align}
defined for every Schwartz function $f$ in $\R^{d}$, where $K : \R^{k}\setminus\Set{0}\to\C$ is a kernel
satisfying the following conditions:
\begin{enumerate}
\item the size condition, i.e.\ there exists a constant $C_K>0$ such
that
\begin{equation}
\label{eq:size}
\abs{K(x)}\le C_K\abs{x}^{-k},\quad \text{for all}\quad x\in\R^k;
\end{equation}
\item the cancellation condition
\begin{align}
\label{eq:cancel1}
\int_{\Omega_{R}\setminus \Omega_{r}} K(y) \dif y = 0,
\quad \text{ for }\quad
0<r<R<\infty;
\end{align}
\item the smoothness condition
\begin{equation}
\label{eq:smoothness}
\sup_{R>0} \sup_{\abs{y}\leq Rt/2} \int_{R\le\abs{x}\le2R}
\abs{K(x)-K(x+y)}\dif x
\leq
\omega_K(t),
\end{equation}
for every $t\in(0, 1)$
with some modulus of continuity $\omega_K$.
\end{enumerate}
In many applications
it is easy to verify the somewhat stronger pointwise version of the smoothness
estimate from \eqref{eq:smoothness}. Namely,
\begin{equation}
\label{eq:K-mod}
\abs{K(x)-K(x+y)}
\leq
\omega_K(\abs{y}/\abs{x}) \abs{x}^{-k},
\quad\text{provided that}\quad \abs{y}\leq \abs{x}/2,
\end{equation}
for some modulus of continuity $\omega_K$. One can immediately see that
condition \eqref{eq:K-mod} implies condition \eqref{eq:smoothness}.
Our next result establishes an analogue of the inequality \eqref{eq:41} for the operators in \eqref{eq:47}.
\begin{theorem}
\label{thm:radon-sing}
Suppose that $\norm{\omega_K^{\theta}}_{\logDini}+\norm{\omega_K^{\theta/2}}_{\Dini}<\infty$ for some $\theta\in(0, 1]$. Then
for every $p\in\Set{1+\theta, (1+\theta)'}$ and $f\in L^{p}(\R^{d})$ we have
\begin{equation}
\label{eq:48}
J^{p}_{2}( (\calH_{t}^{P} f)_{t>0} : \R^{d} \to \C)
\lesssim_{d, p}
\norm{f}_{L^{p}},
\end{equation}
where the implicit constant is independent of the coefficients of $P$.
More precisely,
\begin{enumerate}
\item if $\norm{\omega_K^{\theta}}_{\logDini}<\infty$, then
\begin{align}
\label{eq:53}
J^{p}_{2}((\calH_{2^k}f)_{k\in\Z} : \R^d \to \C)
\lesssim
\norm{f}_{L^{p}};
\end{align}
\item if $\norm{\omega_K^{\theta/2}}_{\Dini}<\infty$, then
\begin{equation}
\label{eq:56}
\norm[\Big]{ \bigl( \sum_{k\in\Z} V^{2}(\calH_{t}f : t \in [2^k,2^{k+1}])^{2} \bigr)^{1/2}}_{L^{p}}
\lesssim
\norm{f}_{L^{p}}.
\end{equation}
\end{enumerate}
\end{theorem}
The inequality \eqref{eq:41} was proved in \cite{MR2434308} for the averages $\calM_{t}^P$ over Euclidean balls.
The inequality \eqref{eq:48} was proved in \cite{MR2434308} for monomial curves, i.e.\ in the case $k=1$, $d=2$, $K(y)=y^{-1}$ and $P(x)=(x, x^a)$, where $a>1$.
General polynomials were considered in \cite{MR3681393} (although jump estimates are not explicitly stated in that article they can also be obtained with minor modifications of the proofs).
Multi-dimensional variants of $\calH_{t}^{P}$ were also studied in \cite{MR3681393} under stronger regularity conditions imposed on the kernel $K$.
Inequalities \eqref{eq:41} and \eqref{eq:48} will be used to establish jump inequalities for the discrete analogues of \eqref{eq:40} and \eqref{eq:47} in \cite{arxiv:1809.03803}.

Finally we provide van der Corput integral estimates in Lemma~\ref{lem:vdC} and Proposition~\ref{prop:vdC-multidim}, which have
the feature that permit to handle the oscillatory integrals with
non-smooth amplitudes. Its broader scope will be needed in the proof
of Theorem~\ref{thm:radon-sing}.

\section{An abstract approach to jump inequalities}
\label{sec:gener-prel}
\subsection{Preliminaries}
Let $(X,\calB,\frakm)$ be a $\sigma$-finite measure space endowed with a sequence of linear Littlewood--Paley operators $(S_j)_{j\in\Z}$ satisfying \eqref{eq:3}, \eqref{eq:4}.
Assume that $(T_t)_{t\in\bbI}$ is a family of linear operators acting on $L^{1}(X) + L^{\infty}(X)$, where the index set $\bbI$ is a subset of $(0, \infty)$. Under suitable conditions imposed on the family $(T_t)_{t\in\bbI}$ in terms of its interactions with the Littlewood--Paley operators $(S_j)_{j\in\Z}$ as in the introduction, we will study strong uniform jump inequalities
\begin{align}
\label{eq:49}
J^{p}_{2}((T_tf)_{t\in\bbI} : X \to \C)\lesssim\norm{f}_{L^p}
\end{align}
in various ranges of $p$'s, see Theorem~\ref{thm:duo+rubio:max}, Theorem~\ref{thm:duo+rubio:jump} and Theorem~\ref{thm:short-var}.

To avoid further problems with measurability we will always assume that $\bbI$ is countable. Usually $\bbI=\D:=\Set{2^n \given n\in\Z}$ the set of all dyadic numbers or $\bbI=\bbU:=\bigcup_{n\in\Z}2^{-n}\N$ the set of non-negative rational numbers whose denominators in reduced form are powers of $2$. In practice, the countability assumption may be removed if for every $f\in L^{1}(X) + L^{\infty}(X)$ the function $\bbI\ni t\mapsto T_tf(x)$ is continuous for $\frakm$-almost every $x\in X$.
In our applications this will be always our case.

We recall the decomposition into long and short jumps from \cite[Lemma 1.3]{MR2434308}, which tells that for every $\lambda>0$ we have
\begin{align}
\label{eq:8}
\begin{split}
\lambda N_{\lambda}(T_tf(x):t\in\bbI) ^{1/2}&\lesssim
\lambda N_{\lambda/3}(T_tf(x):t\in\D)^{1/2}\\
&+\Big(\sum_{k\in\Z} \big(\lambda N_{\lambda}(T_tf(x):t \in[2^k, 2^{k+1})\cap\bbI)^{1/2}\big)^2\Big)^{1/2}.
\end{split}
\end{align}
In other words the $\lambda$-jump counting function can be dominated by the long jumps (the first term in \eqref{eq:8} with $t\in\D$) and the short jumps (the square function in \eqref{eq:8}).
Similar inequalities hold for the maximal function and for $r$-variations.

We deal with $L^p$ bounds for the long jump counting function corresponding to $T_{t}$ with $t\in\D$ in two ways, similarly to \cite{MR837527}.
The first approach is to find an approximating family of operators (see the family $(P_k)_{k\in\Z}$ in Theorem~\ref{thm:duo+rubio:max}) for which the bound in question is known and control a square function that dominates the error term, see \eqref{eq:6} in Theorem~\ref{thm:duo+rubio:max}.
In our case this method works for positive operators with martingales or related operators as the approximating family.
The second approach is to express $T_{2^k}$ as a telescoping sum
\begin{align}
\label{eq:30}
T_{2^k}f=\sum_{j\ge k}T_{2^j}f-T_{2^{j+1}}f=\sum_{j\ge k}B_jf
\end{align}
and try to deduce bounds in question from the behavior of $B_j=T_{2^{j}}-T_{2^{j+1}}$.
This approach is needed if $T_t$ is a truncated singular integral type operator, see Theorem~\ref{thm:duo+rubio:jump}.
Similar strategies also yield $L^p$ bounds for maximal functions $\sup_{k\in\Z}\abs{T_{2^k}f(x)}$ or $r$-variations $V^r(T_{2^k}f(x): k\in\Z)$.

In order to deal with short jumps we note that the square function on the right-hand side of \eqref{eq:8} is dominated by the square function associated with $2$-variations, which in turn is controlled by a series of square functions
\begin{align}
\label{eq:10}
\begin{split}
\Big(\sum_{k\in\Z}& \big(V^2(T_tf(x) :t \in[2^k, 2^{k+1})\cap\bbI)\big)^2\Big)^{1/2}\\
&\le
\sqrt{2}\sum_{l\ge 0} \Big(\sum_{k\in\Z} \sum_{m = 0}^{2^{l}-1}\abs{(T_{2^k+{2^{k-l}(m+1)}} - T_{2^k+{2^{k-l}m}})f(x)}^2 \Big)^{1/2}.
\end{split}
\end{align}
The square function on the right-hand side of \eqref{eq:10} gives rise to assumption \eqref{eq:sq-L2}.
Inequality \eqref{eq:10} follows from the next lemma with $\frakg(t)=T_{2^k+t}f(x)$ and $r=2$.
\begin{lemma}
\label{lem:lewko-lewko}
Let $r\in[1, \infty)$, $k\in\Z$, and a function $\frakg:[0, 2^{k}]\cap\bbU\to\C$ be given.
Then
\begin{align}
\begin{split}
\label{eq:lewko-lewko}
V^r\big(\frakg(t):t \in[0, 2^{k}]\cap\bbU\big)
&\le2^{\frac{r-1}{r}} \sum_{l\ge 0} \Big( \sum_{m = 0}^{2^{l}-1}\abs{\frakg(2^{k-l}(m+1)) - \frakg(2^{k-l}m)}^r \Big)^{1/r}.
\end{split}
\end{align}
\end{lemma}
The variation norm on the left-hand side of \eqref{eq:lewko-lewko} can be extended to all $t\in[0, 2^k]$ if $\frakg:[0, 2^{k}]\to\C$ is continuous.
Lemma~\ref{lem:lewko-lewko} originates in the paper of Lewko and Lewko \cite{MR2885959}, where it was observed that the $2$-variation norm of a sequence of length $N$ can be controlled by the sum of $\log N$ square functions and this observation was used to obtain a variational version of the Rademacher--Menshov theorem.
Inequality \eqref{eq:lewko-lewko}, essentially in this form, was independently proved by the first author and Trojan in \cite{MR3595493} and used to estimate $r$-variations for discrete Radon transforms. Lemma~\ref{lem:lewko-lewko} has been used in several recent articles on $r$-variations, including \cite{MR3777413}. For completeness we include a proof, which is shorter than the previous proofs.
\begin{proof}[Proof of Lemma~\ref{lem:lewko-lewko}]
Due to monotonicity of $r$-variations
it suffices to prove \eqref{eq:lewko-lewko} with $\bbU_N=\Set{u/2^{N} \given u\in\N \text{ and } 0\le u\le 2^{k+N}}$ in place of
$[0,2^{k}]\cap\bbU$.
Observe that
\[
V^r\big(\frakg(t):t \in \bbU_N\big)=V^r\big(\frakg(t/2^{N}):t \in [0,2^{k+N}]\cap\Z \big).
\]
The proof will be completed if we show that
\begin{align}
\label{eq:11}
V^r\big(\frakg(t):t \in [0,2^{n}]\cap\Z\big)
\leq
2^{1-1/r} \sum_{l=0}^n \Big( \sum_{m = 0}^{2^{n-l}-1}\abs{\frakg(2^{l}(m+1)) - \frakg(2^{l}m)}^r \Big)^{1/r}.
\end{align}
Once \eqref{eq:11} is established we apply it with $\frakg(t/2^{N})$ in place of $\frakg(t)$ and $n=k+N$ and obtain \eqref{eq:lewko-lewko}.
We prove \eqref{eq:11} by induction on $n$. The case $n=0$ is easy to verify.
Let $n\geq 1$ and suppose that the claim is known for $n-1$.
Let $0\leq t_{0} < \dotsb < t_{J} < 2^{n}$ be an increasing sequence of integers.
For $j\in\Set{0,\dotsc,J}$ let $s_{j}\leq t_{j} \leq u_{j}$ be the closest smaller and larger even integer, respectively.
Then
\begin{multline*}
\Bigl( \sum_{j=1}^{J} \abs{\frakg(t_{j})-\frakg(t_{j-1})}^{r} \Bigr)^{1/r}\\
=
\Bigl( \sum_{j=1}^{J} \abs{(\frakg(t_{j})-\frakg(s_{j}))+(\frakg(s_{j})-\frakg(u_{j-1}))+(\frakg(u_{j-1})-\frakg(t_{j-1}))}^{r} \Bigr)^{1/r}\\
\leq
\Bigl( \sum_{j=1}^{J} \abs{\frakg(s_{j})-\frakg(u_{j-1})}^{r} \Bigr)^{1/r}
+
\Bigl( \sum_{j=1}^{J} \abs{(\frakg(t_{j})-\frakg(s_{j}))+(\frakg(u_{j-1})-\frakg(t_{j-1}))}^{r} \Bigr)^{1/r}.
\end{multline*}
In the first term we notice that the sequence $u_{0} \leq s_{1} \leq u_{1} \leq \dotsb$ is monotonically increasing and takes values in $2\N$, so we can apply the induction hypothesis to the function $\frakg(2\cdot)$.
In the second term we use the elementary inequality $(a+b)^{r}\leq 2^{r-1}(a^{r}+b^{r})$ and observe $\abs{t_{j}-s_{j}}\leq 1$, $\abs{t_{j-1}-u_{j-1}}\leq 1$, and $s_{j}\geq u_{j-1}$, so that this is bounded by the $l=0$ summand in \eqref{eq:11}.
\end{proof}

\subsection{Preparatory estimates}

We recall Lemma~\ref{lem:duo+rubio} that deduces a vector-valued inequality from a maximal one. Then we apply it to obtain Lemma~\ref{lem:duo+rubio:square}.
\begin{lemma}[{cf.\ \cite[p.\ 544]{MR837527}}]
\label{lem:duo+rubio}
Suppose that $(X,\calB,\frakm)$ is a $\sigma$-finite measure space and $(M_{k})_{k\in\J}$ is a sequence of linear operators on $L^{1}(X)+L^{\infty}(X)$ indexed by a countable set $\J$. The corresponding maximal operator is defined by
\[
M_{*, \J}f := \sup_{k\in\J} \sup_{\abs{g} \leq \abs{f}} \abs{M_{k}g},
\]
where the supremum is taken in the lattice sense.
Let $q_{0},q_{1} \in [1,\infty]$ and $0\leq \theta \leq 1$ with $\frac12 = \frac{1-\theta}{q_{0}}$ and $q_0\le q_1$.
Let $q_{\theta}\in[q_0, q_1]$ be given by
$\frac{1}{q_{\theta}} = \frac{1-\theta}{q_{0}} + \frac{\theta}{q_{1}} = \frac12 + \frac{1-q_{0}/2}{q_{1}}$.
Then
\[
\norm[\Big]{ \big( \sum_{k\in\J} \abs{M_{k}g_{k}}^{2}\big)^{1/2}}_{L^{q_{\theta}}}
\leq
(\sup_{k\in\J} \norm{M_{k}}_{L^{q_{0}}\to L^{q_{0}}})^{1-\theta} \norm{M_{*, \J}}_{L^{q_{1}}\to L^{q_{1}}}^{\theta}
\norm[\Big]{ \big( \sum_{k\in\J} \abs{g_{k}}^{2}\big)^{1/2}}_{L^{q_{\theta}}}.
\]
\end{lemma}
\begin{proof}
Consider the operator $\tilde M g := (M_{k}g_{k})_{k\in\J}$ acting on sequences of functions $g=(g_k)_{k\in\J}$ in $L^{1}(X)+L^{\infty}(X)$.
By Fubini's theorem
\begin{align*}
\norm{\tilde M g}_{L^{q_{0}}(\ell^{q_{0}})}
&=
\norm[\big]{ \norm{M_{k} g_{k}}_{L^{q_{0}}}}_{\ell^{q_{0}}}\\
&\leq
(\sup_{k\in\J} \norm{M_{k}}_{L^{q_{0}}\to L^{q_{0}}})
\norm[\big]{ \norm{g_{k}}_{L^{q_{0}}}}_{\ell^{q_{0}}}\\
&=
(\sup_{k\in\J} \norm{M_{k}}_{L^{q_{0}}\to L^{q_{0}}})
\norm{g}_{L^{q_{0}}(\ell^{q_{0}})}.
\end{align*}
By definition of the maximal operator
\begin{align*}
\norm{\tilde M g}_{L^{q_{1}}(\ell^{\infty})}
&=
\norm[\big]{\sup_{k\in\J} \abs{M_{k} g_{k}}}_{L^{q_{1}}}\\
&\leq
\norm[\big]{M_{*, \J} (\sup_{k\in\J} \abs{g_{k}})}_{L^{q_{1}}}\\
&\leq
\norm{M_{*, \J}}_{L^{q_{1}}\to L^{q_{1}}}
\norm[\big]{\sup_{k\in\J} \abs{g_{k}}}_{L^{q_{1}}}\\
&=
\norm{M_{*, \J}}_{L^{q_{1}}\to L^{q_{1}}}
\norm{g}_{L^{q_{1}}(\ell^{\infty})}.
\end{align*}
The claim for $q_{\theta}\in[q_0, q_1]$ follows by complex interpolation between $L^{q_{0}}(X; \ell^{q_{0}}(\J))$ and $L^{q_{1}}(X;\ell^{\infty}(\J))$.
\end{proof}
\begin{lemma}
\label{lem:duo+rubio:square}
Suppose that $(X,\calB,\frakm)$ is a $\sigma$-finite measure space with a sequence of operators $(S_{k})_{k\in\Z}$ that satisfy the Littlewood--Paley inequality \eqref{eq:4}. Let $1\leq q_{0} \leq q_{1} \leq 2$ and $L\in\N$ be a positive integer and let $\V_L=\Set{(k, l)\in\Z^2 \given 0\le l\le L-1}$.
Let $(M_{k,l})_{(k, l)\in\V_L}$ be a sequence of operators bounded on $L^{q_{1}}(X)$ such that
\begin{equation}
\label{eq:duo+rubio:off-diag}
\norm[\Big]{ \bigl( \sum_{k\in\Z} \sum_{l=0}^{L-1} \abs{ M_{k,l} S_{k+j} f }^{2} \bigr)^{1/2} }_{L^2}
\leq
a_{j} \norm{f}_{L^2}, \qquad f\in L^{2}(X)
\end{equation}
for some positive numbers $(a_{j})_{j\in\Z}$.
Then for $p=q_{1}$ and for all $f\in L^p(X)$ we have
\begin{align}
\label{eq:1}
\begin{split}
\MoveEqLeft
\norm[\Big]{\big(\sum_{k\in\Z} \sum_{l=0}^{L-1} \abs{ M_{k,l} S_{k+j} f}^2\big)^{1/2}}_{L^p}\\
&\lesssim L^{\frac{2-q_{1}}{2-q_{0}} \frac12}\big(\sup_{(k, l)\in\V_L} \norm{M_{k,l}}_{L^{q_{0}}\to L^{q_{0}}}^{\frac{2-q_{1}}{2-q_{0}} \frac{q_{0}}{2}}\big) \norm{M_{*, \V_L}}_{L^{q_{1}}\to L^{q_{1}}}^{\frac{2-q_{1}}{2}} a_{j}^{\frac{q_{1}-q_{0}}{2-q_{0}}}
\norm{f}_{L^p}.
\end{split}
\end{align}
If $M_{k,l}$ are convolution operators on an abelian group $\GG$, then \eqref{eq:1} also holds for $q_{1} \leq p \leq q_{1}'$.
The implicit constants in the conclusion do not depend on the qualitative bounds that we assume for the operators $M_{k,l}$ on $L^{q_{1}}(X)$.
\end{lemma}
\begin{proof}
First we show \eqref{eq:1}.
In the case $q_{1}=2$ this is identical to the hypothesis \eqref{eq:duo+rubio:off-diag}, so suppose $q_{1}<2$.
Let $\theta$ and $q_{\theta} \in [q_{0},q_{1}]$ be as in Lemma~\ref{lem:duo+rubio}, then by that lemma and Littlewood--Paley inequality \eqref{eq:4} we obtain
\begin{equation}
\label{eq:2}
\begin{split}
\MoveEqLeft
\norm[\Big]{\big(\sum_{k\in\Z} \sum_{l=0}^{L-1} \abs{ M_{k,l} S_{k+j} f}^2\big)^{1/2}}_{L^{q_{\theta}}}\\
&\lesssim
\big(\sup_{(k, l)\in\V_L} \norm{M_{k,l}}_{L^{q_{0}}\to L^{q_{0}}}^{1-\theta}\big) \norm{M_{*, \V_L}}_{L^{q_{1}}\to L^{q_{1}}}^{\theta}
\norm[\Big]{\big(\sum_{k\in\Z} \sum_{l=0}^{L-1} \abs{ S_{k+j}f}^2\big)^{1/2}}_{L^{q_{\theta}}}\\
&\lesssim
L^{1/2} \big(\sup_{(k, l)\in\V_L} \norm{M_{k,l}}_{L^{q_{0}}\to L^{q_{0}}}^{1-\theta}\big) \norm{M_{*, \V_L}}_{L^{q_{1}}\to L^{q_{1}}}^{\theta}
\norm{f}_{L^{q_{\theta}}}.
\end{split}
\end{equation}
Since $q_{\theta} \leq q_{1} < 2$, there is a unique $\nu \in (0,1]$ such that $\frac{1}{q_{1}} = \frac{\nu}{q_{\theta}} + \frac{1-\nu}{2}$.
Substituting the definition of $q_{\theta}$ we obtain $\frac{1}{q_{1}} = \frac{\nu\theta}{q_{1}} + \frac12$.
It follows that
\begin{align*}
1-\theta &= \frac{q_{0}}{2},&
\theta &= \frac{2-q_{0}}{2},&
\nu\theta &= \frac{2-q_{1}}{2},\\
\nu &= \frac{2-q_{1}}{2-q_{0}},&
\nu(1-\theta) &= \frac{2-q_{1}}{2-q_{0}} \frac{q_{0}}{2},&
1-\nu &= \frac{q_{1}-q_{0}}{2-q_{0}}.
\end{align*}
Interpolating \eqref{eq:2} with the hypothesis \eqref{eq:duo+rubio:off-diag} gives the claim \eqref{eq:1} for $p=q_1$.

If $M_{k,l}$ are convolution operators, then by duality the first inequality in \eqref{eq:2} also holds with $q_{\theta}$ replaced by $q_{\theta}'$.
Also, $\frac{1}{q_{1}'} = \frac{\nu}{q_{\theta}'} + \frac{1-\nu}{2}$, so the same argument as before also works for $p=q_{1}'$.
The conclusion for $q_{1} < p < q_{1}'$ follows by complex interpolation.
\end{proof}

\subsection{Long jumps for positive operators}
Suppose now we have a sequence of positive linear operators $(A_k)_{k\in\Z}$ and an approximating family of linear operators $(P_k)_{k\in\Z}$ both acting on $L^1(X)+ L^{\infty}(X)$ such that for every $1<p<\infty$ the maximal lattice operator
\begin{align*}
P_{*}f := \sup_{k\in\Z} \sup_{\abs{g} \leq \abs{f}} \abs{P_{k}g},
\end{align*}
satisfies the maximal estimate
\begin{align}
\label{eq:5}
\norm{P_{*}}_{L^p\to L^p}\lesssim 1.
\end{align}

Theorem~\ref{thm:duo+rubio:max} will be based on a variant of bootstrap argument discussed in the context of differentiation in lacunary directions in \cite{MR0466470}.
These ideas were also used to provide $L^p$ bounds for maximal Radon transforms in \cite{MR837527}.
It was observed by Christ that the argument from \cite{MR0466470} can be formulated as an abstract principle, which was useful in many situations \cite{MR949003} and also in the context of dimension-free estimates \cite{MR828824}.
\begin{theorem}
\label{thm:duo+rubio:max}
Assume that $(X,\calB,\frakm)$ is a $\sigma$-finite measure space endowed with a sequence of linear operators $(S_j)_{j\in\Z}$ satisfying \eqref{eq:3} and \eqref{eq:4}. Given parameters $1\leq q_{0}<q_{1}\leq 2$, let $(A_{k})_{k\in\Z}$ be a sequence of positive linear operators such that $\sup_{k\in\Z}\norm{A_{k}}_{L^{q_{0}}\to L^{q_{0}}} \lesssim 1$.
Suppose that the maximal function $P_{*}$ satisfies \eqref{eq:5} with $p=q_{1}$ and
\begin{align}
\label{eq:6}
\norm[\Big]{ \bigl( \sum_{k\in\Z} \abs{ (A_k-P_k)S_{k+j} f }^{2} \bigr)^{1/2} }_{L^2}
\leq
a_{j} \norm{f}_{L^2}, \qquad f\in L^2(X)
\end{align}
for some positive numbers $(a_{j})_{j\in\Z}$ satisfying $\bfa := \sum_{j\in\Z} a_{j}^{\frac{q_{1}-q_{0}}{2-q_{0}}}<\infty$.

Then for all $f\in L^p(X)$ with $p=q_1$ we have
\begin{align}
\label{eq:7}
\norm[\Big]{ \bigl( \sum_{k\in\Z} \abs{ (A_k-P_k)f }^{2} \bigr)^{1/2} }_{L^p}
\lesssim (1+\bfa^{2/q_{1}})\norm{f}_{L^p}.
\end{align}
In particular
\begin{equation}
\label{eq:duo+rubio:max}
\norm{A_{*}}_{L^p \to L^p}
\lesssim
1+\bfa^{2/q_{1}}.
\end{equation}
If in addition we have the jump inequality
\begin{equation}
\label{eq:P-jump}
J^{p}_{2}((P_{k}f)_{k\in\Z} : X \to \C) \lesssim \norm{f}_{L^{p}},
\end{equation}
then also
\begin{equation}
\label{eq:duo+rubio:jump}
J^{p}_{2}((A_{k}f)_{k\in\Z} : X \to \C)
\lesssim
(1+\bfa^{2/q_{1}}) \norm{f}_{L^{p}}.
\end{equation}
If $A_{k}$ and $P_{k}$ are convolution operators on an abelian group $\GG$ all these implications also hold for $q_{1} \leq p \leq q_{1}'$, and we have the vector-valued estimate
\begin{equation}
\label{eq:long-vector-valued}
\norm[\Big]{ \bigl( \sum_{k\in\Z} \abs{A_{k} f_{k}}^{r} \bigr)^{1/r}}_{L^p}
\lesssim
\norm[\Big]{ \bigl( \sum_{k\in\Z} \abs{f_{k}}^{r} \bigr)^{1/r}}_{L^p}
\end{equation}
in the same range $q_{1} \leq p \leq q_{1}'$ for all $1 \leq r \leq \infty$.
\end{theorem}

A few remarks concerning the assumptions in Theorem~\ref{thm:duo+rubio:max} are in order.
In applications it is usually not difficult to verify the assumption \eqref{eq:6}.
For general operators the most reasonable and efficient way is to apply $TT^{*}$ methods.
However, for convolution operators on $\GG$ assumption \eqref{eq:6} can be verified using Fourier transform methods, which may be simpler than $TT^{*}$ methods.
Let us explain the second approach more precisely when $\GG=\R^d$.
We first have to fix some terminology.

Let $A$ be a $d\times d$ real matrix whose eigenvalues have positive real
part. We set
\begin{align}
\label{eq:9}
t^{A}:=\exp(A\log t), \quad \text{for}\quad t>0.
\end{align}

Let $\frakq$ be a smooth $A$-homogeneous quasi-norm on $\R^{d}$, that is, $\frakq:\R^{d}\to [0,\infty)$ is a continuous function, smooth on $\R^d\setminus\Set{0}$, and such that
\begin{enumerate}
\item $\frakq(x)=0 \iff x=0$;
\item there is $C\ge1$ such that for all $x,y\in\R^{d}$ we have
$\frakq(x+y)\le C(\frakq(x)+\frakq(y))$;
\item $\frakq(t^{A}x) = t\frakq(x)$ for all $t>0$ and $x\in\R^{d}$.
\end{enumerate}

Let also $\frakq_{*}$ be a smooth (away from $0$) $A^{*}$-homogeneous
quasi-norm, where $A^{*}$ is the adjoint matrix to $A$. We only have to find a sequence of Littlewood--Paley projections associated with the quasi-norm $\frakq_{*}$. For this purpose let $\phi_0:[0, \infty)\to [0, \infty)$ be a smooth function such that $0\le \phi_0\le \one_{[1/2, 2]}$ and its dilates $\phi_j(x):=\phi_0(2^jx)$ satisfy
\begin{align}
\label{eq:35}
\sum_{j\in\Z}\phi_j^2=\one_{(0, \infty)}.
\end{align}
For each $j\in\Z$ we define the Littlewood--Paley operator $\tilde{S}_j$ such that $\widehat{\tilde{S}_{j}f} = \psi_{j} \widehat{f}$ corresponds to a smooth function $\psi_{j}(\xi) := \phi_{j}(\frakq_{*}(\xi))$ on $\R^{d}$. By \eqref{eq:35} we see that \eqref{eq:3} holds for $S_j=\tilde{S}_j^2$. Moreover, by \cite[Theorem II.1.5]{MR0440268} we obtain the Littlewood--Paley inequality \eqref{eq:4} for the operators $S_j$ and $\tilde{S}_j$.

If $(\Phi_t: t>0)$ is a family of Schwartz functions such that $\widehat{\Phi}_t(\xi)=\widehat{\Phi}(t\frakq_*(\xi))$, where
$\Phi$ is a non-negative Schwartz function on $\R^d$ with integral one, then by \cite[Theorem 1.1]{MR2434308} we know that for every $1<p<\infty$ we have
\begin{align}
\label{eq:36}
J^{p}_{2}((\Phi_{2^k}*f)_{k\in\Z} : \R^d \to \C)
\lesssim
\norm{f}_{L^{p}}, \qquad f\in L^p(\R^d).
\end{align}
The maximal version of inequality \eqref{eq:36} has been known for a long time and follows from the Hardy--Littlewood maximal theorem \cite{MR1232192}. Hence taking $P_kf=\Phi_{2^k}*f$ for $k\in\Z$, we may assume that \eqref{eq:P-jump} is verified.

Suppose now we have a family $(A_k)_{k\in\Z}$ of convolution operators $A_kf=\mu_{2^k}*f$ corresponding to a family of probability measures $(\mu_{t}: t>0)$ on $\R^d$ such that
\begin{align}
\abs{\widehat\mu_{t}(\xi)-\widehat\mu_{t}(0)}
&\leq
\omega(t\frakq_{*}(\xi))
&\text{if } \quad t\frakq_*(\xi)\le1, \label{eq:42}\\
\abs{\widehat\mu_{t}(\xi)}
&\leq
\omega((t\frakq_{*}(\xi))^{-1})
& \text{if } \quad t\frakq_*(\xi)\ge1,\label{eq:43}
\end{align}
for some modulus of continuity $\omega$. 

Theorem~\ref{thm:duo+rubio:max}, taking into account all the facts mentioned above, yields
\begin{align}
\label{eq:16}
J^{p}_{2}((\mu_{2^k}*f)_{k\in\Z} : \R^d \to \C)
\lesssim
\norm{f}_{L^{p}}, \qquad f\in L^p(\R^d)
\end{align}
for $p=q_1$ and $q_0=1$ as long as $\bfa = \sum_{j\in\Z}\omega(2^{-\abs{j}})^{\frac{q_{1}-q_{0}}{2-q_{0}}}<\infty$, since \eqref{eq:6} can be easily verified with $a_j=\omega(2^{-\abs{j}})$ using \eqref{eq:42}, \eqref{eq:43} and the properties of ${S}_j$ and $\Phi$.

\begin{proof}[Proof of Theorem~\ref{thm:duo+rubio:max}]
We begin with the proof of \eqref{eq:7}. If $q_1=2$ then we use \eqref{eq:3} and \eqref{eq:6} and we are done. We now assume that $q_1<2$. By the monotone convergence theorem it suffices to consider only finitely many $M_k:=A_{k}-P_{k}$'s in \eqref{eq:7}, let us say those with $\abs{k} \leq K$.
Restrict all summations and suprema to $\abs{k}\leq K$ and let $B$ be the smallest implicit constant for which \eqref{eq:7} holds with $p=q_1$.
In view of the qualitative boundedness hypothesis we obtain $B<\infty$, but the bound may depend on $K$.
Our aim is to show that $B\lesssim 1+\bfa^{2/q_1}$. There is nothing to do if $B\lesssim 1$. Therefore, we will assume that $B\gtrsim1$, so by \eqref{eq:3}, \eqref{eq:5} and \eqref{eq:1} with $L=1$ and $M_{k, 0}:=M_k$, we obtain
\[
\norm[\Big]{\big(\sum_{\abs{k} \leq K} \abs{ M_{k} f}^2\big)^{1/2}}_{L^p}
\le
\sum_{j\in\Z}\norm[\Big]{\big(\sum_{\abs{k}\leq K} \abs{ M_{k} S_{k+j}f}^2\big)^{1/2}}_{L^p}
\lesssim
\big(1+\norm{M_{*}}_{L^p\to L^p}^{\frac{2-q_{1}}{2}} \bfa\big)\norm{f}_{L^p}.
\]
By positivity we have $\abs{A_{*}f} \leq \sup_{\abs{k}\leq K} A_{k} \abs{f}$ and consequently, we obtain
\begin{align}
\label{eq:12}
\abs{A_{*} f}
\leq
\sup_{\abs{k}\leq K} A_{k} \abs{f}
\leq
\sup_{\abs{k}\leq K} P_{k} \abs{f} + \Big( \sum_{\abs{k}\leq K} \abs{M_{k} \abs{f}}^{2} \Big)^{1/2}.
\end{align}
By \eqref{eq:12} and \eqref{eq:5} we get
\[
\norm{M_{*}}_{L^p\to L^p}\le \norm{P_{*}}_{L^p\to L^p}+\norm{A_{*}}_{L^p\to L^p}\le 2\norm{P_{*}}_{L^p\to L^p}+B\lesssim 1+B.
\]
Taking into account these inequalities we have
\begin{align*}
\norm[\Big]{\big(\sum_{\abs{k} \leq K} \abs{ (A_{k}-P_{k}) f}^2\big)^{1/2}}_{L^p}
=
\norm[\Big]{\big(\sum_{\abs{k} \leq K} \abs{ M_{k} f}^2\big)^{1/2}}_{L^p}
\lesssim
\big(1+\bfa(1+B)^{\frac{2-q_{1}}{2}}\big)\norm{f}_{L^p}.
\end{align*}
Taking the supremum over $f$ gives
\[
B
\lesssim
1 + \bfa (1+B)^{\frac{2-q_{1}}{2}} \lesssim
(1+\bfa) B^{\frac{2-q_{1}}{2}},
\]
since we have assumed $B\gtrsim 1$,
and the conclusion \eqref{eq:7} follows.

Once \eqref{eq:7} is proven then in view of \eqref{eq:12} we immediately obtain \eqref{eq:duo+rubio:max}. In a similar way, if \eqref{eq:P-jump} holds, we deduce \eqref{eq:duo+rubio:jump} from \eqref{eq:7}. Indeed,
\begin{align*}
J^{p}_{2}((A_{k}f)_{k\in\Z})
&\lesssim
J^{p}_{2}((P_{k}f)_{k\in\Z}) + J^{p}_{2}((M_{k}f)_{k\in\Z})\\
&\lesssim
\norm{f}_{L^p} + \norm[\big]{V^{2}(M_{k}f : k\in\Z)}_{L^{p}}\\
&\lesssim
\norm{f}_{L^p} + \norm[\Big]{ \big( \sum_{k\in\Z} \abs{M_{k}f}^{2} \big)^{1/2} }_{L^{p}}.
\end{align*}

In the case of convolution operators we can run the above proof of \eqref{eq:7} with $p=q_{1}'$, since in this case Lemma~\ref{lem:duo+rubio:square} tells that \eqref{eq:1} also holds with $p=q_{1}'$.
Once the estimate \eqref{eq:7} is known for $p=q_{1},q_{1}'$, by interpolation we extend it to $q_{1} \leq p \leq q_{1}'$, and all other inequalities follow as before.
Finally, the vector-valued estimate \eqref{eq:long-vector-valued} with $r=\infty$ is equivalent to the maximal estimate by positivity, with $r=1$ it follows by duality, and with $1<r<\infty$ by complex interpolation.
\end{proof}

\subsection{Long jumps for non-positive operators}
We now drop the positivity assumption and we will be working with general operators $(B_{k})_{k\in\Z}$ acting on $L^1(X)+ L^{\infty}(X)$. This will require some knowledge about the maximal lattice operator $B_{*}$ defined in \eqref{eq:22} and about the sum of $B_k$'s over $k\in\Z$.
No bootstrap argument seems to be available for non-positive operators and therefore additional assumptions like \eqref{eq:23} and \eqref{eq:24} will be indispensable.
The proof of Theorem~\ref{thm:duo+rubio:jump} is based on the ideas from \cite{MR837527}.

\begin{theorem}
\label{thm:duo+rubio:jump}
Assume that $(X,\calB,\frakm)$ is a $\sigma$-finite measure space endowed with a sequence of linear operators $(S_j)_{j\in\Z}$ satisfying \eqref{eq:3} and \eqref{eq:4}. Let $1\leq q_{0}<q_{1}\leq 2$ and let $(B_{k})_{k\in\Z}$ be a sequence of linear operators commuting with the sequence $({S}_j)_{j\in\Z}$ such that $\sup_{k\in\Z}\norm{B_{k}}_{L^{q_{0}}\to L^{q_{0}}} \lesssim 1$. Suppose that the maximal lattice operator
\begin{align}
\label{eq:22}
B_{*}f := \sup_{k\in\Z} \sup_{\abs{g} \leq \abs{f}} \abs{B_{k}g},
\end{align}
satisfies
\begin{align}
\label{eq:23}
\norm{B_{*}}_{L^{q_1}\to L^{q_1}}\lesssim 1.
\end{align}
We also assume
\begin{align}
\label{eq:14}
\norm[\Big]{ \bigl( \sum_{k\in\Z} \abs{ B_kS_{k+j} f }^{2} \bigr)^{1/2} }_{L^2}
\leq
a_{j} \norm{f}_{L^2}, \qquad f\in L^2(X)
\end{align}
for some positive numbers $(a_{j})_{j\in\Z}$.
\begin{enumerate}
\item Suppose that $(B_{k})_{k\in\Z}$ additionally satisfies
\begin{align}
\label{eq:24}
\norm[\Big]{\sum_{k\in\Z}B_k}_{L^{q_1}\to L^{q_1}}\lesssim 1.
\end{align}
Let $P_k:=\sum_{j>k}S_j$ and assume that the jump inequality \eqref{eq:P-jump} holds for the sequence $(P_k)_{k\in\Z}$ with $p=q_1$.
Then for all $f\in L^p(X)$ with $p=q_1$ we have
\begin{align}
\label{eq:15}
\begin{split}
\MoveEqLeft
\qquad J^{p}_{2}\Big(\big(\sum_{j\ge k}B_{j}f\big)_{k\in\Z} : X \to \C\Big)\\
&\lesssim \Big( \norm[\Big]{\sum_{k\in\Z}B_k}_{L^{q_1}\to L^{q_1}}
+
\big(\sup_{k\in\Z} \norm{B_{k}}_{L^{q_{0}}\to L^{q_{0}}}^{\frac{2-q_{1}}{2-q_{0}} \frac{q_{0}}{2}}\big) \norm{B_{*}}_{L^{q_{1}}\to L^{q_{1}}}^{\frac{2-q_{1}}{2}} \tilde{\bfa}\Big)\norm{f}_{L^{p}},
\end{split}
\end{align}
where $\tilde{\bfa} := \sum_{j\in\Z} (\abs{j}+1) a_{j}^{\frac{q_{1}-q_{0}}{2-q_{0}}}<\infty$.
\item Suppose that there is a sequence of self-adjoint linear operators $(\tilde{S}_j)_{j\in\Z}$ such that $S_j=\tilde{S}_j^2$ for every $j\in\Z$ and satisfying \eqref{eq:4} and \eqref{eq:14} with $\tilde{S}_{k+j}$ in place of $S_{k+j}$.
Then for every sequence $(\varepsilon_{k})_{k\in\Z}$ bounded by $1$ and for all $f\in L^p(X)$ with $p=q_1$ we have
\begin{align}
\label{eq:17}
\norm[\Big]{ \sum_{k\in\Z} \varepsilon_k B_kf }_{L^p}
\lesssim \big(\sup_{k\in\Z} \norm{B_{k}}_{L^{q_{0}}\to L^{q_{0}}}^{\frac{2-q_{1}}{2-q_{0}} \frac{q_{0}}{2}}\big) \norm{B_{*}}_{L^{q_{1}}\to L^{q_{1}}}^{\frac{2-q_{1}}{2}} {\bfa}\norm{f}_{L^p},
\end{align}
where ${\bfa}$ is as in Theorem~\ref{thm:duo+rubio:max}.

\end{enumerate}

In the case of convolution operators on an abelian group $\GG$ all these implications also hold for $q_{1} \leq p \leq q_{1}'$.
\end{theorem}

In applications in harmonic analysis we will take $B_k=T_{2^{k}}-T_{2^{k+1}}$ for $k\in\Z$, where $T_t$ is a truncated singular integral operator of convolution type, see \eqref{eq:30}. This class of operators motivates, to a large extent, the assumptions in Theorem~\ref{thm:duo+rubio:jump}. In many cases they can be verified if we manage to find positive operators $A_k$ such that $\abs{B_kf} \lesssim A_k\abs{f}$ for every $k\in\Z$ and $f\in L^{1}(X) + L^{\infty}(X)$. In practice, $A_k$ is an averaging operator. We shall illustrate this more precisely by appealing to the discussion after Theorem~\ref{thm:duo+rubio:max}.

Suppose that $(B_k)_{k\in\Z}$ is a family of convolution operators $B_kf=\sigma_{2^k}*f$ corresponding to a family of finite measures $(\sigma_{t}: t>0)$ on $\R^d$ such that $\sup_{t>0}\norm{\sigma_t}<\infty$ and for every $k\in\Z$ and $t\in[2^k, 2^{k+1}]$ we have
\begin{align}
\abs{\widehat\sigma_{t}(\xi)}
&\leq
\omega(2^k\frakq_{*}(\xi))
&\text{if } \quad 2^k\frakq_*(\xi)\le1, \label{eq:50}\\
\abs{\widehat\sigma_{t}(\xi)}
&\leq
\omega((2^k\frakq_{*}(\xi))^{-1})
& \text{if } \quad 2^k\frakq_*(\xi)\ge1,\label{eq:51}
\end{align}
for some modulus of continuity $\omega$. Additionally, we assume that $\abs{\sigma_{2^k}} \lesssim \mu_{2^k}$ for some family of finite positive measures $(\mu_{t}: t>0)$ on $\R^d$ such that $\sup_{t>0}\norm{\mu_t}<\infty$ and satisfying \eqref{eq:42} and \eqref{eq:43}. In view of these assumptions and Theorem~\ref{thm:duo+rubio:max} we see that condition \eqref{eq:23} holds, since $\abs{B_kf} \lesssim A_k\abs{f}$, where $A_kf=\mu_{2^k}*f$.
Therefore,
\begin{align*}
\norm[\Big]{ \sum_{k\in\Z}B_kf }_{L^p}
\lesssim {\bfa}\norm{f}_{L^p},
\end{align*}
implies \eqref{eq:24} with $p=q_1$ and $q_0=1$, provided that $\bfa = \sum_{j\in\Z} \omega(2^{-\abs{j}})^{\frac{q_{1}-q_{0}}{2-q_{0}}}<\infty$, since \eqref{eq:14} can be verified with $a_j=\omega(2^{-\abs{j}})$ using \eqref{eq:50}, \eqref{eq:51} and the properties of $\tilde{S}_j$ associated with \eqref{eq:35}. Having proven \eqref{eq:23} and \eqref{eq:24} we see that \eqref{eq:15} holds for the operators $B_kf=\sigma_{2^k}*f$
with $p=q_1$ and $q_0=1$ as long as $\tilde{\bfa} = \sum_{j\in\Z} (\abs{j}+1)\omega(2^{-\abs{j}})^{\frac{q_{1}-q_{0}}{2-q_{0}}}<\infty$.

\begin{proof}[Proof of Theorem~\ref{thm:duo+rubio:jump}]
In order to prove inequality \eqref{eq:15} we employ the following decomposition
\begin{align}
\label{eq:19}
\sum_{j\geq k} B_j
=
P_k \sum_{j\in\Z} B_j
-
\sum_{l>0}\sum_{j<0} S_{{k+l}} B_{{k+j}}
+
\sum_{l\le0}\sum_{j\ge0} S_{{k+l}} B_{{k+j}}
\end{align}
(cf.\ \cite[p.\ 548]{MR837527}).
The jump inequality corresponding to the first term on the right-hand side in \eqref{eq:19} is bounded on $L^{p}(X)$ with $p=q_1$,
due to \eqref{eq:P-jump}, and \eqref{eq:24}, which ensures boundedness of the operator $\sum_{j\in\Z} B_{j}$.

The estimates for the second and the third term are similar and we only consider the last term.
We take the $\ell^{2}$ norm with respect to the parameter $k$ and estimate
\begin{align*}
\MoveEqLeft
J^{p}_{2}\Big(\big(\sum_{l\le0}\sum_{j\ge0} B_{{k+j}}S_{k+l}f\big)_{k\in\Z} : X \to \C\Big)\\
&\le\norm[\Big]{\big(\sum_{k\in\Z}\abs{\sum_{l\le0}\sum_{j\ge0} B_{{k+j}}S_{k+l}f}^2\big)^{1/2}}_{L^p}\\
&=
\norm[\Big]{\big(\sum_{k\in\Z}\abs{\sum_{m\geq 0} \sum_{n=k-m}^{k} B_{{n+m}}S_{n} f}^2\big)^{1/2}}_{L^p}\\
&\leq
\sum_{m\geq 0} \norm[\Big]{\big(\sum_{k\in\Z}\abs{\sum_{n=k-m}^{k} B_{{n+m}}S_{n} f}^2\big)^{1/2}}_{L^p}
&&\text{by triangle inequality}\\
&\leq
\sum_{m\geq 0} (m+1)^{1/2} \norm[\Big]{\big(\sum_{k\in\Z} \sum_{n=k-m}^{k} \abs{ B_{{n+m}}S_{n}f}^2\big)^{1/2}}_{L^p}
&&\text{by H\"older's inequality}\\
&=
\sum_{m\geq 0} (m+1) \norm[\Big]{\big(\sum_{n\in\Z} \abs{B_{{n+m}}S_{n}f}^2\big)^{1/2}}_{L^p}.
\end{align*}
By \eqref{eq:1}, with $L=1$ and $M_{k, 0}:=B_k$, we obtain
\begin{align*}
\sum_{j\in\Z} (\abs{j}+1) \norm[\Big]{\big(\sum_{k\in\Z} \abs{ B_{k} S_{k+j} f}^2\big)^{1/2}}_{L^p}
\lesssim \big(\sup_{k\in\Z} \norm{B_{k}}_{L^{q_{0}}\to L^{q_{0}}}^{\frac{2-q_{1}}{2-q_{0}} \frac{q_{0}}{2}}\big) \norm{B_{*}}_{L^{q_{1}}\to L^{q_{1}}}^{\frac{2-q_{1}}{2}} \tilde{\bfa}
\norm{f}_{L^p}.
\end{align*}

To prove the second part observe that for a sequence of functions $(f_j)_{j\in\Z}$ in $L^p(X; \ell^2(\Z))$
we have the following inequality
\begin{align}
\label{eq:20}
\norm[\Big]{\sum_{j\in\Z}\tilde{S}_{j} f_j }_{L^p}
\lesssim
\norm[\Big]{\big(\sum_{j\in\Z} \abs{f_j}^{2} \big)^{1/2}}_{L^p},
\end{align}
which is the dual version of inequality \eqref{eq:4} for the sequence $(\tilde{S}_{j})_{j\in\Z}$. To prove \eqref{eq:17} we will use \eqref{eq:3} and \eqref{eq:20}. Indeed,
\begin{align*}
\norm[\Big]{\sum_{k\in\Z} \varepsilon_{k} B_{k} f}_{L^p}
&\le \sum_{j\in\Z}\norm[\Big]{\sum_{k\in\Z} \varepsilon_{k} B_{k} S_{k+j} f}_{L^p} && \text{by \eqref{eq:3}}\\
&=
\sum_{j\in\Z}\norm[\Big]{\sum_{k\in\Z} \tilde{S}_{k+j}(\varepsilon_{k}B_{k}\tilde{S}_{k+j} f)}_{L^p}&& \text{since $S_j=\tilde{S}_j^2$}\\
&\lesssim\sum_{j\in\Z}\norm[\Big]{\big(\sum_{k\in\Z} \abs{B_{k}\tilde{S}_{k+j} f}^2\big)^{1/2}}_{L^p}&& \text{by \eqref{eq:20}}\\
&\lesssim \big(\sup_{k\in\Z} \norm{B_{k}}_{L^{q_{0}}\to L^{q_{0}}}^{\frac{2-q_{1}}{2-q_{0}} \frac{q_{0}}{2}}\big) \norm{B_{*}}_{L^{q_{1}}\to L^{q_{1}}}^{\frac{2-q_{1}}{2}} {\bfa}
\norm{f}_{L^p},
\end{align*}
where in the last step we have used Lemma~\ref{lem:duo+rubio:square}, with $L=1$ and $M_{k, 0}:=B_k$.
\end{proof}

\subsection{Short variations}
We will work with a sequence of linear operators $(A_t)_{t\in\bbU}$ (not necessarily positive) acting on $L^1(X)+ L^{\infty}(X)$. However, positive operators will be distinguished in our proof and in this case we can also proceed as before using some bootstrap arguments.

For every $k\in\Z$ and $t\in[2^k, 2^{k+1}]$ we will use the following notation
\[
\Delta((A_s)_{s\in\bbI})_tf:=\Delta(A_t)f:=A_tf-A_{2^k}f.
\]

\begin{theorem}
\label{thm:short-var}
Assume that $(X,\calB,\frakm)$ is a $\sigma$-finite measure space endowed with a sequence of linear operators $(S_j)_{j\in\Z}$ satisfying \eqref{eq:3} and \eqref{eq:4}.
Let $(A_{t})_{t\in\bbU}$ be a family of linear operators such that the square function estimate
\begin{equation}
\label{eq:sq-L2}
\norm[\Big]{ \bigl( \sum_{k\in\Z}\sum_{m=0}^{2^l-1} \abs{(A_{2^k+{2^{k-l}(m+1)}} - A_{2^k+{2^{k-l}m}}) S_{j+k} f}^{2} \bigr)^{1/2}}_{L^2}
\le
2^{-\frac{l}{2}} a_{j, l}\norm{f}_{L^2}
\end{equation}
holds for all $j\in\Z$ and $l\in\N$ with some numbers $a_{j, l}\ge0$ such that for every $0<\varepsilon<\rho$ we have
\begin{align}
\label{eq:21}
\sum_{l\ge0}\sum_{j\in\Z}2^{-\varepsilon l} a_{j, l}^{\rho}<\infty.
\end{align}
\begin{enumerate}
\item\label{thm:short-var:case-vv}
Let $1<q_0<2$ and $4<q_{\infty}<\infty$,
and suppose that for each $q_0 \leq p \leq q_{\infty}$ the vector-valued estimate
\begin{equation}
\label{eq:short-long-square-hypothesis}
\norm[\Big]{ \big( \sum_{k\in\Z} \abs{A_{2^k(1+t)}f_{k}}^{2} \big)^{1/2} }_{L^p}
\lesssim
\norm[\Big]{ \big( \sum_{k\in\Z} \abs{f_{k}}^{2} \big)^{1/2} }_{L^p}
\end{equation}
holds uniformly in $t\in\bbU\cap[0, 1]$.
Then for each $\frac{3}{1+1/q_0} <p <\frac{4}{1+2/q_{\infty}}$ we have
\begin{equation}
\label{eq:short-var}
\norm[\Big]{ \bigl( \sum_{k\in\Z} V^{2}(A_{t}f : t \in [2^k,2^{k+1}] \cap \bbU)^{2} \bigr)^{1/2}}_{L^{p}}
\lesssim
\norm{f}_{L^{p}},
\end{equation}
and for each $4 \leq p < q_{\infty}$ and $r>\frac{p}{2}\frac{q_{\infty}-2}{q_{\infty}-p}$ we have
\begin{equation}
\label{eq:short-r-var}
\norm[\Big]{ \bigl( \sum_{k\in\Z} V^{r}(A_{t}f : t \in [2^k,2^{k+1}] \cap \bbU)^{r} \bigr)^{1/r}}_{L^{p}}
\lesssim
\norm{f}_{L^{p}}
\end{equation}
for all $f\in L^p(X)$.
\item\label{thm:short-var:case-holder}
Let $q_{0} \in [1,2)$ and $\alpha \in [0,1]$ be such that $\alpha q_{0} \leq 1$.
Suppose that we have the operator norm H\"older type condition
\begin{equation}
\label{eq:short-Holder}
\norm{A_{t+h}-A_{t}}_{L^{q_{0}}\to L^{q_{0}}} \lesssim \bigg(\frac{h}{t}\bigg)^{\alpha},
\quad
t,t+h \in \bbU, \text{ and } \ h\in(0, 1].
\end{equation}
Then for every exponent $q_1$ satisfying
\begin{equation}
\label{eq:26}
q_0\le 2 - \frac{2-q_{0}}{2-\alpha q_{0}} < q_1 \leq 2,
\end{equation}
and such that
\begin{equation}
\label{eq:short-long-max-hypothesis-pos}
\norm{\Delta((A_s)_{s\in\bbU})_{*,\bbU}}_{L^{q_1}\to L^{q_1}} \lesssim 1
\end{equation}
we have for all $f\in L^p(X)$ with $p=q_1$ that the estimate \eqref{eq:short-var} holds with the implicit constant which is a constant multiple of
\[
\bfa:=\sum_{l\ge0}\sum_{j\in\Z}2^{-\big(
\alpha \frac{2-q_{1}}{2-q_{0}} \frac{q_{0}}{2}
+
\frac{1}{2}\frac{q_{1}-q_{0}}{2-q_{0}}-\frac{2-q_{1}}{2-q_{0}} \frac12\big)
l} a_{j, l}^{\frac{q_{1}-q_{0}}{2-q_{0}}}<\infty.
\]

\item\label{thm:short-var:case-holder-positive} Moreover, if $(A_{t})_{t\in\bbU}$ is a family of positive linear operators, then the condition \eqref{eq:short-long-max-hypothesis-pos} may be replaced by a weaker condition
\begin{equation}
\label{eq:short-long-max-hypothesis}
\norm{A_{*,\D}}_{L^{q_1}\to L^{q_1}} \lesssim 1
\end{equation}
and the estimate \eqref{eq:short-var} holds as well with the implicit constant which is a constant multiple of $1+\bfa^{2/q_1}$.
\end{enumerate}
In the case of convolution operators on an abelian group $\GG$ the implication from \eqref{eq:short-long-max-hypothesis} to \eqref{eq:short-var} also holds with $p$ replaced by $p'$.

\end{theorem}

Theorem~\ref{thm:short-var} combined with the results formulated in the previous two paragraphs for dyadic scales will allow us to control, in view of \eqref{eq:8}, the cases for general scales. The first part of Theorem~\ref{thm:short-var} gives \eqref{eq:short-var} in a restricted range of $p$'s. If one asks for a larger range, a smoothness condition like in \eqref{eq:short-Holder} must be assumed. Inequality \eqref{eq:short-Holder} combined with maximal estimate \eqref{eq:short-long-max-hypothesis-pos} gives larger range of $p$'s in \eqref{eq:short-var}. If we work with a family of positive operators the condition \eqref{eq:short-long-max-hypothesis-pos} may be relaxed to \eqref{eq:short-long-max-hypothesis} by some bootstrap argument. In the context of discussion after Theorem~\ref{thm:duo+rubio:max} and Theorem~\ref{thm:duo+rubio:jump} let us look at a particular situation of (\ref{thm:short-var:case-holder}) and prove \eqref{eq:short-var}.

Suppose that $(A_t)_{t>0}$ is a family of convolution operators $A_tf=\sigma_{t}*f$ corresponding to a family of finite measures $(\sigma_{t}: t>0)$ on $\R^d$ such that $\sup_{t>0}\norm{\sigma_t}<\infty$ and satisfying \eqref{eq:50} and \eqref{eq:51}. We assume that $\abs{\sigma_{t}}\lesssim \mu_{t}$ for some family of finite positive measures $(\mu_{t}: t>0)$ on $\R^d$ such that $\sup_{t>0}\norm{\mu_t}<\infty$ and satisfying \eqref{eq:42} and \eqref{eq:43} to make sure that condition \eqref{eq:short-long-max-hypothesis-pos} holds. Additionally, let us assume that \eqref{eq:short-Holder} holds with $\alpha=1$ and $q_0=1, 2$. By Plancherel's theorem, \eqref{eq:50} and \eqref{eq:51} we obtain
\begin{align}
\label{eq:45}
\norm{(A_{2^k+{2^{k-l}(m+1)}} - A_{2^k+{2^{k-l}m}})S_{j+k}f}_{L^2}\lesssim \omega(2^{-\abs{j}})\norm{S_{j+k}f}_{L^2}.
\end{align}
Thus \eqref{eq:short-Holder} with $q_0=2$, $t=2^k+{2^{k-l}m}$, $h=2^{k-l}$ combined with \eqref{eq:45} imply
\begin{align}
\label{eq:52}
\norm{(A_{2^k+{2^{k-l}(m+1)}} - A_{2^k+{2^{k-l}m}})S_{j+k}f}_{L^2}\lesssim \min(2^{-l},\omega(2^{-\abs{j}}))\norm{S_{j+k}f}_{L^2}.
\end{align}
Consequently \eqref{eq:sq-L2} holds with $a_{j, l}=\min\Set{1,2^l\omega(2^{-\abs{j}})}$ and Theorem~\ref{thm:short-var} gives the desired conclusion as long as $\bfa=\sum_{l\ge0}\sum_{j\in\Z}2^{-\frac{(q_{1}-1)l}{2}} (\min\Set{1,2^l\omega(2^{-\abs{j}})})^{q_{1}-1}<\infty.
$

\begin{proof}[Proof of Theorem~\ref{thm:short-var}: case (\ref{thm:short-var:case-vv})]
By Minkowski's inequality for $2 \leq s \leq q_{\infty} < \infty$ we have
\begin{align*}
\MoveEqLeft
\norm[\Big]{ \bigl( \sum_{k\in\Z}\sum_{m=0}^{2^l-1} \abs{(A_{2^k+{2^{k-l}(m+1)}} - A_{2^k+{2^{k-l}m}}) f_{k}}^{s} \bigr)^{1/s}}_{L^{q_{\infty}}}^{s}\\
&=
\norm[\Big]{\sum_{m=0}^{2^l-1} \sum_{k\in\Z} \abs{(A_{2^k+{2^{k-l}(m+1)}} - A_{2^k+{2^{k-l}m}}) f_{k}}^{s}}_{L^{{q_{\infty}}/s}}\\
&\leq
\sum_{m=0}^{2^l-1} \norm[\Big]{\sum_{k\in\Z} \abs{(A_{2^k+{2^{k-l}(m+1)}} - A_{2^k+{2^{k-l}m}}) f_{k}}^{s}}_{L^{{q_{\infty}}/s}}\\
&\leq
2^l \sup_{0\leq m < 2^l} \norm[\Big]{\bigl( \sum_{k\in\Z} \abs{(A_{2^k+{2^{k-l}(m+1)}} - A_{2^k+{2^{k-l}m}}) f_{k}}^{s} \bigr)^{1/s}}_{L^{q_{\infty}}}^{s}\\
&\leq
2^{l+s} \sup_{0\leq m \leq 2^l} \norm[\Big]{\bigl( \sum_{k\in\Z} \abs{A_{2^k+{2^{k-l}m}} f_{k}}^{2} \bigr)^{1/2}}_{L^{q_{\infty}}}^{s}\\
&\lesssim
2^{l+s} \norm[\Big]{\bigl( \sum_{k\in\Z} \abs{f_{k}}^{2} \bigr)^{1/2}}_{L^{q_{\infty}}}^{s},
\end{align*}
where we have applied \eqref{eq:short-long-square-hypothesis} in the last step.
Using this with $f_{k}=S_{j+k}f$ and applying \eqref{eq:4} we obtain
\[
\norm[\Big]{ \bigl( \sum_{k\in\Z}\sum_{m=0}^{2^l-1} \abs{(A_{2^k+{2^{k-l}(m+1)}} - A_{2^k+{2^{k-l}m}}) S_{j+k} f}^{s} \bigr)^{1/s}}_{L^{q_{\infty}}}
\lesssim
2^{l/s} \norm{f}_{L^{q_{\infty}}}
\]
for all $2 \leq s \leq q_{\infty} < \infty$.
By interpolation with \eqref{eq:sq-L2} we obtain
\begin{align}
\label{eq:27}
\norm[\Big]{ \bigl( \sum_{k\in\Z}\sum_{m=0}^{2^l-1} \abs{(A_{2^k+{2^{k-l}(m+1)}} - A_{2^k+{2^{k-l}m}}) S_{j+k} f}^{r} \bigr)^{1/r}}_{L^p}
\lesssim
2^{-\frac{\theta l}{2}+\frac{(1-\theta) l}{s}} a_{j, l}^{\theta} \norm{f}_{L^p},
\end{align}
where $0 < \theta \leq 1$ and
$\frac{1}{r} = \frac{\theta}{2} + \frac{1-\theta}{s}$ and
$\frac{1}{p} = \frac{\theta}{2} + \frac{1-\theta}{q_{\infty}}$, so
$\theta=\frac{2}{p}\frac{q_{\infty}-p}{q_{\infty}-2}$. By Lemma~\ref{lem:lewko-lewko} or more precisely by an analogue of
inequality \eqref{eq:10} with $\ell^r$ norm in place of $\ell^2$ norm
and by \eqref{eq:27} we obtain
\begin{equation}
\label{eq:28}
\norm[\Big]{ \bigl( \sum_{k\in\Z} V^{r}(A_{t}f : t \in [2^k,2^{k+1}] \cap \bbU)^{r} \bigr)^{1/r}}_{L^{p}}
\lesssim
\sum_{l\ge0}\sum_{j\in\Z}2^{-\frac{\theta l}{2}+\frac{(1-\theta) l}{s}} a_{j, l}^{\theta} \norm{f}_{L^p}.
\end{equation}
In view of \eqref{eq:21} with
$\varepsilon=\frac{\theta }{2}-\frac{(1-\theta) }{s}$ and $\rho=\theta$ this estimate is summable in $l$ and $j$, provided that
$-\theta/2 + (1-\theta)/s < 0$.
In particular, for $2\leq p < \frac{4}{1+2/q_{\infty}}$ we use $s=2$.
For $4\leq p < q_{\infty}$ we use $s>\frac{q_{\infty}(p-2)}{q_{\infty}-p}$ and then $r>\frac{p}{2}\frac{q_{\infty}-2}{q_{\infty}-p}$.

For $q_0 \in (1,2)$ by Minkowski's inequality and \eqref{eq:short-long-square-hypothesis} we have
\begin{align*}
\MoveEqLeft
\norm[\Big]{ \bigl( \sum_{k\in\Z}\sum_{m=0}^{2^l-1} \abs{(A_{2^k+{2^{k-l}(m+1)}} - A_{2^k+{2^{k-l}m}}) f_{k}}^{2} \bigr)^{1/2}}_{L^{q_0}}\\
&\leq
\sum_{m=0}^{2^l-1} \norm[\Big]{ \big( \sum_{k\in\Z} \abs{(A_{2^k+{2^{k-l}(m+1)}} - A_{2^k+{2^{k-l}m}}) f_{k}}^{2} \big)^{1/2}}_{L^{q_0}}\\
&\leq
2^{l+1} \sup_{0\leq m \leq 2^l} \norm[\Big]{\bigl( \sum_{k\in\Z} \abs{A_{2^k+{2^{k-l}m}} f_{k}}^{2} \bigr)^{1/2}}_{L^{q_0}}\\
&\lesssim
2^l \norm[\Big]{\bigl( \sum_{k\in\Z} \abs{f_{k}}^{2} \bigr)^{1/2}}_{L^{q_0}}.
\end{align*}
Substituting $f_{k}=S_{j+k}f$, applying \eqref{eq:4}, and
interpolating with \eqref{eq:sq-L2} we obtain
\begin{align}
\label{eq:29}
\norm[\Big]{ \bigl( \sum_{k\in\Z}\sum_{m=0}^{2^l-1} \abs{(A_{2^k+{2^{k-l}(m+1)}} - A_{2^k+{2^{k-l}m}}) S_{j+k} f}^{2} \bigr)^{1/2}}_{L^p}
\lesssim
2^{-\frac{\theta l}{2}+(1-\theta) l} a_{j, l}^{\theta} \norm{f}_{L^p},
\end{align}
with $\frac{1}{p} = \frac{\theta}{2} + \frac{1-\theta}{q_{0}}$, for
$0<\theta<1$. Hence $\theta=\frac{2}{p}\frac{p-q_{0}}{2-q_{0}}$ and in
view of \eqref{eq:21} with $\varepsilon=\frac{\theta }{2}-(1-\theta) $ and $\rho=\theta$ this estimate is summable in $l$ and $j$, provided that
$-\theta/2 + (1-\theta) < 0$. The conclusion again follows from
Lemma~\ref{lem:lewko-lewko} and \eqref{eq:29} like in
\eqref{eq:28} with $\frac{3}{1+1/q_0} < p \leq 2$.
\end{proof}

\begin{proof}[Proof of Theorem~\ref{thm:short-var}: case (\ref{thm:short-var:case-holder}) and case (\ref{thm:short-var:case-holder-positive})]
By the monotone convergence theorem we may restrict $k$ in \eqref{eq:short-var} to $\abs{k} \leq K_{0}$ and parameters $t$ to the set $\bbU_{L_0}^k:=\Set{u/2^{L_0} \given u\in\N \text{ and } 2^{k+L_0}\le u\le 2^{k+L_0+1}}$ for some $K_{0} \in \N$ and $L_{0} \in \Z$ as long as we obtain estimates independent of $K_{0}$ and $L_{0}$. Fix $K_{0}, L_{0}$ and let $\bbI := \bigcup_{\abs{k}\le K_0}\bbU_{L_0}^k$. Let $q_{1}$ satisfy \eqref{eq:26} then invoking \eqref{eq:3} and \eqref{eq:1}, with $L=2^l$, we obtain
\begin{align*}
\MoveEqLeft
\norm[\Big]{ \bigl( \sum_{\abs{k} \leq K_{0}}\sum_{m=0}^{2^l-1} \abs{(A_{2^k+{2^{k-l}(m+1)}} - A_{2^k+{2^{k-l}m}}) f}^{2} \bigr)^{1/2}}_{L^p}\\
&\lesssim
2^{\frac{2-q_{1}}{2-q_{0}} \frac{l}{2}} \big(\sup_{\substack{\abs{k}\le K_0,\\
0\le m<2^l}} \norm{A_{2^k+{2^{k-l}(m+1)}} - A_{2^k+{2^{k-l}m}}}_{L^{q_{0}}\to L^{q_{0}}}^{\frac{2-q_{1}}{2-q_{0}} \frac{q_{0}}{2}}\big) \norm{\Delta((A_{s})_{s\in\bbU})_{*,\bbI}}_{L^{q_{1}}\to L^{q_{1}}}^{\frac{2-q_{1}}{2}}\\
&\qquad\cdot
\Bigl( \sum_{j\in\Z} (2^{-\frac{l}{2}} a_{j, l})^{\frac{q_{1}-q_{0}}{2-q_{0}}} \Bigr)
\norm{f}_{L^p}\\
&\lesssim
2^{\frac{2-q_{1}}{2-q_{0}} \frac{l}{2}} \big((2^{-\alpha l})^{\frac{2-q_{1}}{2-q_{0}} \frac{q_{0}}{2}}\big) \norm{\Delta((A_{s})_{s\in\bbU})_{*,\bbI}}_{L^{q_{1}}\to L^{q_{1}}}^{\frac{2-q_{1}}{2}}
2^{-\frac{l}{2}\frac{q_{1}-q_{0}}{2-q_{0}}}\sum_{j\in\Z} a_{j, l}^{\frac{q_{1}-q_{0}}{2-q_{0}}}\norm{f}_{L^p}.
\end{align*}
In order for the right-hand side to be summable in $l$ we need
\[
\frac{2-q_{1}}{2-q_{0}} \frac12
-
\alpha \frac{2-q_{1}}{2-q_{0}} \frac{q_{0}}{2}
-
\frac{1}{2}\frac{q_{1}-q_{0}}{2-q_{0}}
<
0
\]
\[
\iff
(2-q_{1})
-
\alpha (2-q_{1}) q_{0}
-
(q_{1}-q_{0})
<
0.
\]
It suffices to ensure
\[
(2-q_{1}) (1-\alpha q_{0})
-
(q_{1}-q_{0})
<
0
\]
\[
\iff
q_{1} > \frac{2 (1-\alpha q_{0}) + q_{0}}{2-\alpha q_{0}} = 2 - \frac{2-q_{0}}{2-\alpha q_{0}},
\]
and this is our hypothesis \eqref{eq:26}. Hence under this condition by Lemma~\ref{lem:lewko-lewko} we conclude for general operators that
\begin{equation}
\label{eq:short-var:restricted}
\begin{split}
\MoveEqLeft
\norm[\Big]{ \bigl( \sum_{|k|\leq K_{0}} V^{2}(A_{t}f : t\in \bbU_{L_0}^k)^{2}\bigr)^{1/2}}_{L^p}\\
&\lesssim
\sum_{l = 0}^{K_0+L_0} \norm[\Big]{ \bigl( \sum_{\abs{k} \leq K_{0}} \sum_{m=0}^{2^l-1} \abs{(A_{2^k+{2^{k-l}(m+1)}} - A_{2^k+{2^{k-l}m}})f}^{2}\bigr)^{1/2}}_{L^p}\\
&\lesssim
\norm{\Delta((A_{s})_{s\in\bbU})_{*,\bbI}}_{L^{q_{1}}\to L^{q_{1}}}^{\frac{2-q_{1}}{2}}
\bfa\norm{f}_{L^p},
\end{split}
\end{equation}
as desired. For positive operators crude estimates and interpolation show that
\[
B := \norm{A_{*,\bbI}}_{L^p\to L^p} < \infty
\]
with $p=q_{1}$, since $\bbI$ is finite. Note that
\begin{align}
\label{eq:13}
\sup_{t\in\bbI}\abs{A_tf(x)}\le\sup_{t\in\D}\abs{A_tf(x)}+\Big(\sum_{k\in\Z}\sup_{t\in[2^k, 2^{k+1})\cap\bbI}\abs{(A_t-A_{2^k})f(x)}^2\Big)^{1/2}.
\end{align}
Therefore, appealing to \eqref{eq:13}, \eqref{eq:short-long-max-hypothesis} and \eqref{eq:short-var:restricted} we obtain
by a bootstrap argument that $B \lesssim 1 + B^{\frac{2-q_{1}}{2}}\bfa$, since
\[
\norm{\Delta((A_{s})_{s\in\bbU})_{*,\bbI}}_{L^{q_{1}}\to L^{q_{1}}}^{\frac{2-q_{1}}{2}} \lesssim B^{\frac{2-q_{1}}{2}}.
\]
Hence, $B \lesssim 1+\bfa^{2/q_1}$.
In particular, the estimate \eqref{eq:short-var:restricted} becomes uniform in $\bbI \subset \bbU$, and this simultaneously implies \eqref{eq:short-var}.

In the case of convolution operators we may replace $p=q_{1}$ by $p=q_{1}'$ in Lemma~\ref{lem:duo+rubio:square} and all subsequent arguments.
\end{proof}

\section{Applications}
\subsection{Dimension-free estimates for jumps in the continuous setting}
\label{sec:convex}
We begin by providing dimension-free endpoint estimates, for $r=2$, in the main results of \cite{MR3777413}.
Let $G\subset\R^{d}$ be a symmetric convex body. By definition of the averaging operator \eqref{eq:av-op} we have $\calA_{t}^{G} \tilde U = \tilde U \calA_{t}^{U(G)}$, where $\tilde Uf := f \circ U$ is the composition operator
with an invertible linear map $U : \R^{d} \to \R^{d}$.
It follows that all estimates in Section~\ref{sec:introduction} are not affected if $G$ is replaced by $U(G)$.

By \cite{MR868898}, after replacing $G$ by its image under a suitable invertible linear transformation, we may assume that the normalized characteristic function $\mu := \abs{G}^{-1} \one_{G}$ satisfies
\begin{align}
\label{eq:muhat-decay}
\abs{\widehat{\mu}(\xi)}
&\leq C \abs{\xi}^{-1},\\
\label{eq:muhat-near0}
\abs{\widehat{\mu}(\xi)-1}
&\leq C \abs{\xi},\\
\label{eq:muhat-smooth}
\abs{\langle \xi, \nabla \widehat{\mu}(\xi)\rangle}
&\leq C
\end{align}
with the constant $C$ independent of the dimension.
In \cite{MR868898} these estimates were proved with $\abs{L(G) \xi}$ in place of $\abs{\xi}$ on the right-hand side, where $L(G)$ is the isotropic constant corresponding to $G$. The above form is obtained by rescaling.

Then $\calA_{t} := \calA_{t}^{G}$ is the convolution operator with $\mu_t$ and $\widehat{\mu}_{t}(\xi) = \widehat{\mu}(t\xi)$.
The Poisson semigroup is defined by
\[
\widehat{\calP_t f}(\xi)
:=
p_t(\xi)\widehat{f}(\xi),
\quad \text{where}\quad
p_t(\xi) := e^{-2\pi t\abs{\xi}}.
\]
The associated Littlewood--Paley operators are given by $S_{k} := \calP_{2^{k}} - \calP_{2^{k+1}}$.
Their Fourier symbols satisfy
\begin{equation}
\label{eq:Sk-symbols}
\abs{ \widehat{S}_{k}(\xi) } \lesssim \min\Set{ 2^{k}\abs{\xi}, 2^{-k}\abs{\xi}^{-1} },
\end{equation}
where $\widehat{S}_{k}(\xi)$ is the multiplier associated with the operator $S_k$, i.e. $\widehat{S_{k}f}(\xi)=\widehat{S}_{k}(\xi)\widehat{f}(\xi)$. From now on, for simplicity of notation, we will use this convention.
The symbols associated with the Poisson semigroup $P_{k} := \calP_{2^{k}}$ satisfy
\begin{equation}
\label{eq:Pk-symbols}
\abs{ \widehat{P}_{k}(\xi) - 1 } \lesssim \abs{2^{k}\xi}, \quad\text{and}\quad
\abs{ \widehat{P}_{k}(\xi) } \lesssim 2^{-k}\abs{\xi}^{-1}.
\end{equation}

\begin{proof}[Proof of Theorem~\ref{thm:body-dyadic}]
We verify that the sequence $(A_{k})_{k\in\Z}$, where $A_k:=\calA_{2^k}$ satisfies the hypotheses of Theorem~\ref{thm:duo+rubio:max} for every $1 = q_{0} < q_{1} \leq 2$.

The maximal inequality \eqref{eq:5} and the Littlewood--Paley inequality \eqref{eq:4} for the Poisson semigroup with constants independent of the dimension are well-known \cite{MR0252961}.
The jump estimate \eqref{eq:P-jump} was recently established in \cite[Theorem 1.5]{arxiv:1808.04592}.

It remains to verify condition \eqref{eq:6} for the operators $M_{k}:=A_{k}-P_{k}$.
In view of \eqref{eq:muhat-decay}, \eqref{eq:muhat-near0} and \eqref{eq:Pk-symbols}, we have
\[
\abs{\widehat{M}_{k}(\xi)}
\lesssim
\min\Set{\abs{2^{k} \xi}^{-1}, \abs{2^{k}\xi}}.
\]
For $\xi \in \R^{d} \setminus \Set{0}$ let $k_{0} \in \Z$ be such that $\tilde\xi = 2^{k_{0}}\xi$ satisfies $\abs{\tilde\xi} \simeq 1$.
By \eqref{eq:Pk-symbols} it follows that
\begin{equation}
\label{eq:off-diagonal-FT}
\begin{split}
\sum_{k\in\Z} \abs{\widehat{M}_{k}(\xi) \widehat{S}_{k+j}(\xi)}^{2}
&\lesssim
\sum_{k\in\Z} \min\Set{\abs{2^{k} \xi}^{-1}, \abs{2^{k} \xi}}^{2} \min\Set{\abs{2^{k+j} \xi}^{-1}, \abs{2^{k+j} \xi}}^{2}\\
&=
\sum_{k\in\Z} \min\Set{\abs{2^{k} \tilde\xi}^{-1}, \abs{2^{k} \tilde\xi}}^{2} \min\Set{\abs{2^{k+j} \tilde\xi}^{-1}, \abs{2^{k+j} \tilde\xi}}^{2}\\
&\lesssim
\sum_{k\in\Z} \min\Set{2^{-k} , 2^{k} }^{2} \min\Set{(2^{k+j})^{-1}, 2^{k+j}}^{2}\\
&\lesssim
2^{-\delta\abs{j}}
\end{split}
\end{equation}
for $\delta \in (0,2)$ with the implicit constant independent of the dimension.
By Plancherel's theorem this shows that \eqref{eq:6} holds with $a_{j} \lesssim 2^{-\delta \abs{j}/2}$.
\end{proof}

\begin{proof}[Proof of Theorem~\ref{thm:body-short}]
We will apply Theorem~\ref{thm:short-var} with
$A_t:=\calA_t:=\calA_t^G$. By a simple scaling we have
$\calA_{2^k(1+t)}=\calA_{2^k}^{(1+t)G}$. Hence Theorem~\ref{thm:duo+rubio:max}, with $A_k=\calA_{2^k}^{(1+t)G}$,
applies and we obtain the vector-valued
inequality~\eqref{eq:long-vector-valued} for all $1<p<\infty$ and $r=2$, which consequently guarantees \eqref{eq:short-long-square-hypothesis}. It remains to verify the hypothesis \eqref{eq:sq-L2} of Theorem~\ref{thm:short-var}.
We repeat the estimate \cite[(4.23)]{MR3777413}.
By \eqref{eq:muhat-smooth} for $t>0$ and $h>0$ we have
\begin{equation}
\label{eq:55}
\abs[\big]{\widehat{\mu}\big((t+h)\xi\big) - \widehat{\mu}\big(t\xi\big)}
\leq
\int_{t}^{t+h}
\abs{\langle \xi, \nabla \widehat{\mu}(u\xi)\rangle}\dif u
\lesssim
\int_{t}^{t+h}
\frac{\dif u}{u}
\lesssim \frac{h}{t}.
\end{equation}
By the Plancherel theorem this implies
\begin{equation}
\label{eq:short-Lip}
\norm{\calA_{t+h}-\calA_{t}}_{L^2\to L^2} \lesssim \frac{h}{t}.
\end{equation}
This allows us to estimate the square of the left-hand side of \eqref{eq:sq-L2} by
\begin{align*}
LHS\eqref{eq:sq-L2}^{2} &=
\sum_{k \in \Z} \sum_{m=0}^{2^l-1} \norm{(\calA_{2^k+{2^{k-l}(m+1)}} - \calA_{2^k+{2^{k-l}m}}) S_{j+k}f}_{L^2}^{2}\\
&\lesssim
\sum_{k \in \Z} \sum_{m=0}^{2^l-1} 2^{-2l} \norm{ S_{j+k} f}_{L^2}^{2}\\
&=
2^{-l} \sum_{k \in \Z} \norm{ S_{j+k} f}_{L^2}^{2}\\
&\lesssim
2^{-l} \norm{f}_{L^2}^{2}.
\end{align*}

Secondly, by \eqref{eq:muhat-decay} and \eqref{eq:muhat-near0} for every $0\leq m < 2^l$ we have
\[
\abs[\big]{\widehat{\mu}((2^k+{2^{k-l}(m+1)})\xi) - \widehat{\mu}((2^k+{2^{k-l}m})\xi)}
\lesssim
\min \Set{\abs{2^k\xi},\abs{2^k\xi}^{-1}}.
\]
Arguing similarly to \eqref{eq:off-diagonal-FT} we obtain
\begin{align*}
LHS\eqref{eq:sq-L2}^{2}
\lesssim
2^l 2^{-\delta\abs{j}} \norm{f}_{2}^{2}.
\end{align*}
Hence \eqref{eq:sq-L2} holds with $a_{j, l}=\min\Set{1,2^l2^{-\delta\abs{j}/2}}$.
\end{proof}

\begin{proof}[Proof of Theorem~\ref{thm:lq-ball-short}]
By Theorem~\ref{thm:body-dyadic} we have the hypothesis \eqref{eq:short-long-max-hypothesis} of Theorem~\ref{thm:short-var}.
The hypothesis \eqref{eq:sq-L2} was verified in the proof of Theorem~\ref{thm:body-short}.
The remaining hypothesis \eqref{eq:short-Holder} is given by \cite[Lemma 4.2]{MR3777413}, but we give a more direct proof.

Recall that $B^q$ is the unit ball induced by $\ell^q$ norm in $\R^d$. From \cite{MR1042048} (for $1\leq q<\infty$), and \cite{MR3273441} (for $q=\infty$) we use the multiplier norm estimate
\[
\norm{\tilde m}_{M^{p}} \lesssim_{p,q,\alpha} 1,
\quad
\tilde m = (\xi\cdot\nabla)^{\alpha} \widehat{\mu}
\]
for $\alpha\in(0, 1)$ and $p\in (1,\infty)$ with implicit constant independent of the dimension.
For a Lipschitz function $h : (1/2,\infty) \to \R$ such that $\abs{h(t)} \lesssim \abs{t}^{-1}$ and $\abs{h'(t)}\lesssim \abs{t}^{-1}$ fractional differentiation can be inverted by fractional integration:
\[
h(t) = \frac{1}{\Gamma(\alpha)} \int_{t}^{+\infty} (u-t)^{\alpha-1} D^{\alpha}h(u) \dif u,
\quad t>1/2,
\]
see \cite[Lemma 6.9]{arxiv:1602.02015}.
In particular, for $t>1$ we obtain
\[
h(t)-h(1)
=
\frac{1}{\Gamma(\alpha)} \int_{1}^{+\infty} ((u-t)_{+}^{\alpha-1} - (u-1)^{\alpha-1}) D^{\alpha}h(u) \dif u,
\]
where $u_{+} := \max(u,0)$ denotes the positive part.
In view of \eqref{eq:muhat-decay} and \eqref{eq:muhat-smooth} this result can be applied to the function $h(t) = \widehat{\mu}(t\xi)$ for any $\xi\in \R^{d}\setminus\Set{0}$.
Observing $D^{\alpha} h(u) = u^{-\alpha} \tilde m(u\xi)$ we obtain
\[
\widehat{\mu}(t\xi)-\widehat{\mu}(\xi)
=
\frac{1}{\Gamma(\alpha)} \int_{1}^{+\infty} ((u-t)_{+}^{\alpha-1} - (u-1)_{+}^{\alpha-1}) u^{-\alpha} \tilde m(u\xi) \dif u.
\]
On the other hand we have
\[
\int_{1}^{+\infty} \abs{(u-t)_{+}^{\alpha-1} - (u-1)_{+}^{\alpha-1}}{u^{-\alpha}} \dif u
\lesssim_{\alpha}
(t-1)^{\alpha},
\]
and for a Schwartz function $f \in \calS(\R^{d})$ this implies
\begin{align*}
\MoveEqLeft
\norm{\FT^{-1}_{\xi}((\widehat{\mu}(t\xi)-\widehat{\mu}(\xi))\widehat{f}(\xi))}_{L^p}\\
&\leq
\int_{1}^{+\infty} \abs{(u-t)_{+}^{\alpha-1} - (u-1)_{+}^{\alpha-1}} u^{-\alpha} \cdot
\norm{\FT^{-1}_{\xi}(((u\xi\cdot \nabla)^{\alpha}\widehat{\mu})(\xi)\widehat{f}(\xi))}_{L^p} \dif u\\
&\lesssim_{\alpha}
(t-1)^{\alpha} \sup_{u>0}
\norm{\FT^{-1}_{\xi}(((u\xi\cdot \nabla)^{\alpha}\widehat{\mu})(u\xi)\widehat{f}(\xi))}_{L^p}\\
&\lesssim_{\alpha}
(t-1)^{\alpha} \norm{((\xi\cdot \nabla)^{\alpha}\widehat{\mu})(\xi)}_{M^{p}} \norm{f}_{L^p},
\end{align*}
where we have used the Fourier inversion formula and Fubini's theorem in the first step and scale invariance of the multiplier norm in the last step.
Since the multiplier $\widehat{\mu}(t\xi)-\widehat{\mu}(\xi)$ is (qualitatively) bounded on $L^{p}$ with norm $\leq 2$, by density of Schwartz functions this implies
\[
\norm{\widehat{\mu}(t\cdot)-\widehat{\mu}}_{M^{p}} \lesssim_{\alpha} (t-1)^{\alpha},
\]
which by scaling implies the hypothesis \eqref{eq:short-Holder}.
\end{proof}
Finally we emphasize that once Theorem~\ref{thm:body-dyadic} is
proved, alternative proofs of Theorem~\ref{thm:body-short} an Theorem
\ref{thm:lq-ball-short} follow by appealing to the short variational
estimates given in \cite{MR3777413}.
\subsection{Dimension-free estimates for jumps in the discrete setting}
We briefly outline the proof of Theorem~\ref{jump-discrete}. The strategy is much the same as for the proof of Theorem~\ref{thm:body-dyadic} and Theorem~\ref{thm:body-short}. Let
\begin{align*}
\frakm_N(\xi)=\frac{1}{(2N+1)^d}\sum_{m\in Q_{N}}e^{2\pi i m \cdot\xi
},\qquad \text{for}\qquad \xi\in\T^d
\end{align*}
be the multiplier corresponding to the operators $\bfA_N$ defined in \eqref{eq:37}.
Here we remind the reader of the following estimates for $\frakm_N$ established recently in \cite{arxiv:1804.07679}. Namely there is a constant $0<C<\infty$ independent of the dimension such the for every $N, N_1, N_2\in\N$ and for every $\xi\in\T^d\equiv[-1/2, 1/2)^d$ we have
\begin{align}
\label{eq:54}
\begin{split}
\abs{\frakm_N(\xi)}&\le C(N\abs{\xi})^{-1},\\
\abs{\frakm_N(\xi)-1}&\le CN\abs{\xi},\\
\abs{\frakm_{N_1}(\xi)-\frakm_{N_2}(\xi)}&\le C\abs{N_1-N_2}\max\Set[\big]{N_1^{-1},N_2^{-1}},
\end{split}
\end{align}
where $\abs{\cdot}$ denotes the Euclidean norm restricted to $\T^d$.

The discrete Poisson semigroup is defined by
\[
\widehat{\calP_t f}(\xi)
:=
p_t(\xi)\widehat{f}(\xi),
\quad \text{where}\quad
p_t(\xi) := e^{-2\pi t\abs{\xi}_{\rm sin}},
\]
for every $\xi\in\T^d$ and
\begin{equation*}
\abs{\xi}_{\rm sin}:=\Big(\sum_{j=1}^d (\sin (\pi \xi_j))^2\Big)^{1/2}.
\end{equation*}
We set $P_{k} := \calP_{2^{k}}$ and the associated Littlewood--Paley operators are given by $S_{k} := \calP_{2^{k}} - \calP_{2^{k+1}}$.
The maximal inequality \eqref{eq:5} and the Littlewood--Paley inequality \eqref{eq:4} for the discrete Poisson semigroup with constants independent of the dimension follow from \cite{MR0252961}. The jump estimate \eqref{eq:P-jump} for discrete Poisson semigroup was recently proved in \cite[Theorem 1.5]{arxiv:1808.04592}.
Moreover, using $\abs{\xi}\leq \abs{\xi}_{\rm sin}\leq \pi\abs{\xi}$ for $\xi\in\T^d$, we see that the corresponding Fourier symbols $\widehat{S}_{k}(\xi)$ and $\widehat{P}_{k}(\xi)$ satisfy estimates \eqref{eq:Sk-symbols} and \eqref{eq:Pk-symbols} as well.

In order to prove \eqref{eq:39} we have to verify that the sequence $(A_{k})_{k\in\N}$, where $A_k:=\bfA_{2^k}$ satisfies the hypotheses of Theorem~\ref{thm:duo+rubio:max} for every $1 = q_{0} < q_{1} \leq 2$. Taking into account \eqref{eq:54}, \eqref{eq:Sk-symbols} and \eqref{eq:Pk-symbols} (associated with the discrete Poisson semigroup) it suffices to proceed as in the proof of Theorem~\ref{thm:body-dyadic}. To prove \eqref{eq:38} we argue as in the proof of Theorem~\ref{thm:body-short}.

\subsection{Jump inequalities for the operators of Radon type}
\label{sec:radon}
In this section we prove Theorem~\ref{thm:radon-av} and Theorem~\ref{thm:radon-sing}. By the lifting procedure for the Radon transforms described in \cite[Chapter 11, Section 2.4]{MR1232192} we can assume without loss of generality that our polynomial mapping $P(x):=(x)^{\Gamma}$ is the canonical polynomial mapping for some $\Gamma \subset \N_{0}^{k} \setminus \Set{0}$ with lexicographical order, given by
\[
\R^k \ni x = (x_{1},\dotsc,x_{k}) \mapsto (x)^{\Gamma}:= (x_1^{\gamma_1} \dotsm x_k^{\gamma_k} : \gamma\in\Gamma) \in \R^{\Gamma},
\]
where $\R^{\Gamma}:=\R^{\card{\Gamma}}$ is identified with the space of all vectors whose coordinates are labeled by multi-indices
$\gamma=(\gamma_1, \dotsc, \gamma_k)\in\Gamma$.

Throughout what follows $A$ is the diagonal $\card{\Gamma} \times \card{\Gamma}$ matrix such that $(A x)_\gamma = \abs{\gamma} x_\gamma$ for every $x\in\R^{\Gamma}$ and let $\frakq_*$ be the quasi-norm associated with $A^{*}=A$, given by
\begin{align*}
\frakq_*(\xi)=\max_{\gamma\in\Gamma}\big(\abs{\xi_{\gamma}}^{\frac{1}{\abs{\gamma}}}\big), \quad
\text{for}\quad \xi\in\R^{\Gamma}.
\end{align*}
We shall later freely appeal, without explicit mention, to the discussions after
Theorem~\ref{thm:duo+rubio:max}, Theorem~\ref{thm:duo+rubio:jump} and Theorem~\ref{thm:short-var}
with $d=\card{\Gamma}$, $A$ and $\frakq_*$ as above.
\begin{proof}[Proof of Theorem~\ref{thm:radon-av}]
Let $\calM_{t}:=\calM_{t}^P$, where $P(x)=(x)^{\Gamma}$. Observe that $\calM_t$ is a convolution operator with a probability measure $\mu_t$, whose Fourier transform is defined by
\begin{align*}
\widehat{\mu}_t(\xi):=\frac{1}{\abs{\Omega_{t}}}\int_{\Omega_{t}} e^{-2\pi i \xi \cdot
(y)^{\Gamma}}\dif y,\quad \text{for}\quad \xi\in\R^{\Gamma}.
\end{align*}
Condition \eqref{eq:43} with $\omega(t)=t^{1/d}$ follows from Proposition~\ref{prop:vdC-multidim} and Lemma~\ref{lem:boundary-of-convex}.
It is not difficult to see that \eqref{eq:42} also holds.

In order to prove \eqref{eq:41} it suffices, in view of \eqref{eq:8}, to show inequality \eqref{eq:duo+rubio:jump} with $A_k:=\calM_{2^k}$ and inequality \eqref{eq:short-var} with $A_t:=\calM_{t}$ for every $1 = q_{0} < q_{1} \leq 2$. We have already seen that \eqref{eq:16} holds, hence \eqref{eq:duo+rubio:jump} holds and we are done.
We now show \eqref{eq:short-var}. For this purpose note that \eqref{eq:short-Holder} holds for all $1\le q_0<\infty$. This combined with \eqref{eq:42} and \eqref{eq:43} permits us to prove \eqref{eq:45} and \eqref{eq:52}, which imply \eqref{eq:sq-L2} and Theorem~\ref{thm:short-var} yields the conclusion.
\end{proof}

\begin{proof}[Proof of Theorem~\ref{thm:radon-sing}]
Let $\calH_{t}:=\calH_{t}^P$, where $P(x)=(x)^{\Gamma}$. Denote the Fourier multiplier corresponding to the truncated singular Radon transform by
\begin{equation}
\label{eq:FT-cont-sing}
\Psi_t(\xi) := \int_{\R^k\setminus \Omega_{t}} e^{-2\pi i \xi \cdot
(y)^{\Gamma}} K(y) \dif y,\quad \text{for}\quad \xi\in\R^{\Gamma}.
\end{equation}

For a fixed $\kappa\in (0,1)$ we claim
\begin{equation}
\label{eq:psi-decay-infty}
\begin{split}
\abs{ \Psi_{t}(\xi)-\Psi_{s}(\xi) }
&\lesssim_{\kappa}
\abs{ t^A \xi }_{\infty}^{-1/d} + \omega_K(\abs{ t^A \xi }_{\infty}^{-1/d})\\
&\lesssim (t\frakq_*(\xi))^{-1/d}+\omega_K((t\frakq_*(\xi))^{-1/d}),\quad \text{if}\quad t\frakq_*(\xi)\ge1,
\end{split}
\end{equation}
for all $s,t\in (0,\infty)$ such that $\kappa t\leq s\leq t$. Indeed, by Proposition~\ref{prop:vdC-multidim} we obtain
\begin{align*}
\abs{ \Psi_{t}(\xi)-\Psi_{s}(\xi) }&=\abs[\bigg]{\int_{\Omega_{t}\setminus \Omega_{s}} e^{-2\pi i \xi \cdot
(y)^{\Gamma}} K(y) \dif y}\\
&\lesssim
\sup_{v\in\R^{k} : \abs{v} \leq t\Lambda^{-1/d}} \int \abs{(\ind{\Omega_t\setminus\Omega_s}K)(y)-(\ind{\Omega_t\setminus\Omega_s} K)(y-v)} \dif y
\end{align*}
with $\Lambda=\sum_{\gamma\in\Gamma}t^{\abs{\gamma}}\abs{\xi_{\gamma}}$.
The claim \eqref{eq:psi-decay-infty} clearly holds for $\Lambda\leq 1$.
If $\Lambda\ge1$, then for a fixed $v$ we use \eqref{eq:smoothness} and the fact that
$\Omega_t\setminus\Omega_s\subseteq B(0, t)\setminus B(0,c_{\Omega}\kappa t)$
to estimate the contribution of $y$ such that $y,y-v \in \Omega_t\setminus\Omega_s$.
On the set of $y$ such that exactly one of $y,y-v$ is contained in $\Omega_t\setminus\Omega_s$ we use \eqref{eq:size}; the measure of this set is bounded by a multiple of $t^{k-1} \abs{v}$ due to Lemma~\ref{lem:boundary-of-convex}.
This finishes the proof of \eqref{eq:psi-decay-infty}.

Additionally, we have
\begin{align}
\label{eq:25}
\abs{\Psi_{t}(\xi)-\Psi_{s}(\xi)}
\lesssim
\abs{t^{A} \xi}_{\infty}^{1/d}\lesssim (t\frakq_*(\xi))^{1/d}+\omega_K((t\frakq_*(\xi))^{1/d}),\quad \text{if}\quad t\frakq_*(\xi)\le1
\end{align}
due to the cancellation condition \eqref{eq:cancel1} and \eqref{eq:size}.

To prove \eqref{eq:48} we fix $\theta\in(0, 1]$ and $p\in\Set{1+\theta, (1+\theta)'}$ and invoking
\eqref{eq:8} it suffices to prove inequalities \eqref{eq:53} and \eqref{eq:56}.
Inequality \eqref{eq:53} will follow from \eqref{eq:15} with $q_0=1$, $q_1=1+\theta$ and $B_j:=\calH_{2^j}-\calH_{2^{j+1}}$ upon expressing $\calH_{2^k}$ as a telescoping series like in \eqref{eq:30}. Inequality \eqref{eq:56} will be a consequence of \eqref{eq:short-var} with $q_0=1$, $q_1=1+\theta$ and $A_t:=\calH_{t}$. Let $(\sigma_{t}: t>0)$ be a family of measures
defined by
\begin{align}
\label{eq:223}
\sigma_{t}*f(x)
=
\int_{\Omega_{t}\setminus \Omega_{2^k}} f(x-(y)^{\Gamma}) K(y) \dif
y,\quad \text{for every}\quad t\in[2^k, 2^{k+1}],
\quad k\in\Z.
\end{align}

Estimates \eqref{eq:psi-decay-infty} and \eqref{eq:25} allow us to verify \eqref{eq:50} and \eqref{eq:51} respectively with $\omega(t):=t^{1/d}+\omega_K(t^{1/d})$. Moreover $\abs{\sigma_{2^k}}\lesssim \mu_{2^{k}}$, where $\mu_{t}$ is the measure associated with the averaging operator $\calM_{t}$. Hence the discussion after Theorem~\ref{thm:duo+rubio:jump} guarantees that inequality \eqref{eq:15} holds, since $B_kf=\sigma_{2^{k+1}}*f$. To prove
\eqref{eq:short-var} it suffices to note that \eqref{eq:short-Holder} holds for all $1\le q_0<\infty$. Moreover inequalities \eqref{eq:45} and \eqref{eq:52} remain true for $A_t=\calH_{t}$. Then Theorem~\ref{thm:short-var} completes the proof.
\end{proof}

\appendix

\section{Neighborhoods of boundaries of convex sets}
\label{sec:convex-boundary}
We will show how to control the measure of neighborhoods of the boundaries of convex sets.
The proof of the lemma below is based on a simple Vitali covering argument.
\begin{lemma}
\label{lem:boundary-of-convex}
Let $\Omega\subset\R^{k}$ be a bounded and convex set and let $0<s\lesssim \diam(\Omega)$.
Then
\[
\meas{\Set{ x \in \R^{k} \given \dist(x,\partial\Omega) < s }}
\lesssim_{k}
s \diam(\Omega)^{k-1}.
\]
The implicit constant depends only on the dimension $k$, but not on the convex set $\Omega$.
\end{lemma}
\begin{proof}
Let $r=\diam\Omega$. By translation we may assume $\Omega\subseteq B(0,r)$, where $B(y,s)$ denotes an open ball centered at $y\in\R^k$ with radius $s>0$. Notice
\[
\Set{ x \in \R^{k} \given \dist(x,\partial\Omega) < s }\subseteq \bigcup_{y\in\partial\Omega} B(y,s).
\]
By the Vitali covering lemma there exists a finite subset $Y\subset\partial\Omega$ such that the balls $B(y,s)$ with $y\in Y$, are
pairwise disjoint and
\[
\meas[\Big]{\bigcup_{y\in\partial\Omega} B(y,s)} \lesssim \meas[\Big]{\bigcup_{y\in Y} B(y,s)}.
\]
Consider the nearest-point projection $P : \R^{k}\to \cl \Omega$, that is,
$P(x)=x'$, where $x'\in \cl \Omega$ is the unique point such that $\abs{x-x'}=\dist(x, \cl \Omega)$. It is well-known that $P$ is well-defined and contractive with respect to the Euclidean metric. The restriction of
$P$ to the sphere $\partial B(0,r)$ defines a surjection
$P_{\partial} : \partial B(0,r) \to \partial\Omega$. This follows from
the fact that for every point $x\in\partial\Omega$ there exists a
linear functional $\phi:\R^{k}\to\R$ such that $\phi(y)\leq \phi(x)$
for every
$y\in \cl \Omega$, see e.g.\ \cite[Corollary 11.6.1]{MR0274683}). For each
$y\in Y$ we choose $z(y)\in \partial B(0,r)$ such that $P_{\partial}(z(y))=y$.
Then the balls $B(z(y),s)$ are pairwise disjoint in view of
the contractivity of $P$ and contained in the set
\[
\Set{ x\in\R^k \given r-s < \abs{x} < r+s}
\]
that has measure $\lesssim s(r+s)^{k-1}$. But the union of the balls
$B(z(y),s)$ has the same measure as $\bigcup_{y\in Y} B(y,s)$, and the
conclusion follows.
\end{proof}

\section{Estimates for oscillatory integrals}
\label{sec:vdc}
We present the following variant of van der Corput's oscillatory integral
lemma with a rough amplitude function.
\begin{lemma}
\label{lem:vdC}
Given an interval $(a,b)\subset\R$ suppose that $\phi : (a,b)\to\R$ is a smooth function
such that $\abs{\phi^{(k)}(x)}\gtrsim \lambda$ for every $x\in(a, b)$ with some
$\lambda>0$. Assume additionally that
\begin{itemize}
\item either $k\geq 2$,
\item or $k=1$ and $\phi'$ is monotonic.
\end{itemize}
Then for every locally integrable function $\psi : \R\to\C$ we have
\[
\abs[\bigg]{\int_{a}^{b} e^{i \phi(x)} \psi(x) \dif x}
\lesssim_{k}
\inf_{a\leq x\leq b} \int_{x-\lambda^{-1/k}}^{x+\lambda^{-1/k}}\abs{\psi(y)}\dif y
+
\lambda^{1/k} \int_{-\lambda^{-1/k}}^{\lambda^{-1/k}} \int_{a}^{b} \abs{\psi(x)-\psi(x-y)} \dif x\dif y.
\]
\end{lemma}
\begin{proof}
Let $\eta$ be a smooth positive function with $\supp\eta\subseteq
[-1,1]$ and $\int_{\R}\eta(x)\dif x=1$.
Let $\rho(x) := \psi * \lambda^{1/k} \eta(\lambda^{1/k} x)$, and note that
\[
\abs{\psi(x)-\rho(x)} \leq \lambda^{1/k} \int_{\R} \abs{\psi(x)-\psi(x-y)} \abs{\eta(\lambda^{1/k} y)} \dif y.
\]
Then we may replace $\psi$ by $\rho$ on the left-hand side of the conclusion.
For every $x_{0}\in(a, b)$ by partial integration and the van der
Corput lemma, see for example \cite[Section VIII.1.2]{MR1232192}, we have
\begin{multline*}
\abs[\bigg]{\int_{a}^{b} e^{i \phi(x)} \rho(x) \dif x}
=
\abs[\bigg]{\rho(x_{0}) \int_{a}^{b} e^{i \phi(x)} \dif x
+
\int_{a}^{b} e^{i \phi(x)} \int_{x_{0}}^{x} \rho'(y) \dif y \dif x}\\
\le
\abs[\bigg]{
\rho(x_{0}) \int_{a}^{b} e^{i \phi(x)} \dif x}
+
\abs[\bigg]{
\int_{a}^{x_{0}} \rho'(y) \int_{a}^{y} e^{i \phi(x)} \dif x \dif y}
+
\abs[\bigg]{
\int_{x_{0}}^{b} \rho'(y) \int_{y}^{b} e^{i \phi(x)} \dif x \dif y}\\
\lesssim
\lambda^{-1/k} \bigg( \abs{\rho(x_{0})} + \int_{a}^{b} \abs{\rho'(x)} \dif x \bigg).
\end{multline*}
The latter term is estimated using
\[
\abs{\rho'(x)}
=
\abs{(\psi(x)-\psi) * \lambda^{1/k} \eta(\lambda^{1/k} \cdot)'(x)}
\lesssim
\lambda^{2/k} \int_{\R} \abs{\psi(x)-\psi(x-y)} \abs{\eta'(\lambda^{1/k}y)} \dif y,
\]
and the conclusion follows.
\end{proof}
We will also need a multidimensional version of Lemma~\ref{lem:vdC}. As before $B(y,s)$ denotes an open ball centered at $y\in\R^k$ with radius $s>0$.

\begin{proposition}[{\cite{arxiv:1711.03524}}]
\label{prop:vdC-multidim}
Given $d, k\in\N$, let
$P(x) = \sum_{1\leq \abs{\alpha} \leq d} \lambda_{\alpha} x^{\alpha}$
be a polynomial in $k$ variables of degree at most $d$ with real coefficients.
Let $R>0$ and let $\psi: \R^k\to\C$ be an integrable function supported in $B(0,R/2)$.
Then
\[
\abs[\bigg]{\int_{\R^{k}} e^{i P(x)} \psi(x) \dif x}
\lesssim_{d, k}
\sup_{v\in\R^{k} : \abs{v} \leq R\Lambda^{-1/d}} \int_{\R^{k}} \abs{\psi(x)-\psi(x-v)} \dif x,
\]
where $\Lambda := \sum_{1\leq \abs{\alpha} \leq d} R^{\abs{\alpha}}\abs{\lambda_{\alpha}}$.
\end{proposition}
We include the proof for completeness.
\begin{proof}
Changing the variables we have $\abs[\big]{\int_{\R^k} e^{i P(x)}
\psi(x) \dif x}=R^k\abs[\big]{\int_{\R^k} e^{i P_R(x)} \psi_R(x)
\dif x}$, where $P_R(x)=\sum_{1\leq \abs{\alpha} \leq d}
R^{\abs{\alpha}}\lambda_{\alpha}x^{\alpha}$,
$\psi_R(x)=\psi(Rx)$ and $\supp \psi_R\subseteq B(0, 1/2)$.
Let us define
\[
\beta=\sup_{v\in\R^{k} : \abs{v} \leq \Lambda^{-1/d}} \int_{\R^{k}} \abs{\psi_R(x)-\psi_R(x-v)} \dif x,
\]
and observe that $\norm{\psi_R}_{L^{1}}\lesssim \beta\Lambda^{1/d}$. So
there is nothing to prove if $\Lambda\lesssim 1$. We assume that
$\Lambda\gtrsim1$. Let $\eta$ be a non-negative smooth bump function
with integral $1$, which is
supported in the ball $B(0, 1/2)$. Then we define
$\rho(x)=\Lambda^{k/d}\eta(\Lambda^{1/d}x)$ and
$\phi(x)=\psi_R*\rho(x)$ and we note
\begin{align*}
\int_{\R^k}\abs{\psi_R(x)-\phi(x)}\dif x\le \Lambda^{k/d}
\int_{\R^k}\int_{\R^k}\abs{\psi_R(x)-\psi_R(x-y)}\dif
x\eta(\Lambda^{1/d}y)\dif y\lesssim \beta.
\end{align*}
The proof will be completed if we show that
\begin{align}
\label{eq:198}
\abs[\bigg]{\int_{\R^{k}} e^{i P_R(x)} \phi(x) \dif x}
\lesssim_{d, k}\beta.
\end{align}
Since $\phi$ is a smooth function supported in $B(0, 1)$ we invoke \cite[Lemma 2.2]{MR1879821} to get the conclusion. Indeed,
\cite[Lemma 2.2]{MR1879821} ensures that there exists a unit vector $\xi\in\R^k$
and an integer $m\in\N$ such that
$\abs{(\xi\cdot\nabla)^{m}P_R}>c_{k,d}\Lambda$ on the unit ball $B(0,1)$
for some $c_{k,d}>0$. We may assume, without loss of generality, that
$\xi=e_1=(1, 0,\dotsc,0)\in\R^k$. Then by the van der
Corput lemma, see for example \cite[Corollary p.334]{MR1232192} we obtain
\begin{align*}
\abs[\bigg]{\int_{\R^{k}} e^{i P_R(x)} \phi(x) \dif x}
&\lesssim\Lambda^{-1/d}\int_{\R^{k-1}\cap B(0, 1)}\bigg(\abs{\phi(1,
x')}+\int_{-1}^1\abs{\partial_1\phi(x_1, x')}\dif x_1\bigg)\dif x'\\
&\lesssim \Lambda^{-1/d}\norm{\nabla\phi}_{L^1},
\end{align*}
since $\supp \phi\subseteq B(0, 1)$ and
$\phi(1, x')=0$ for every $x'\in\R^{k-1}\cap B(0, 1)$.

We now show that $\norm{\nabla\phi}_{L^1}\lesssim
\Lambda^{1/d}\beta$. Indeed, for every $j\in\{1, \ldots, k\}$ we have
\begin{align*}
\norm{\partial_j\phi}_{L^1}&=\int_{\R^k}\abs[\Big]{\int_{\R^k}\psi_R(x-y)\partial_j\rho(y)\dif
y}\dif x\\
&=\int_{\R^k}\abs[\Big]{\int_{\R^k}\big(\psi_R(x)-\psi_R(x-y)\big)\partial_j\rho(y)\dif
y}\dif x\\
& \lesssim \Lambda^{k/d+1/d}\int_{\R^k}\int_{\R^k}\abs{\psi_R(x)-\psi_R(x-y)}\abs{(\partial_j\eta)(\Lambda^{1/d}y)}\dif
x\dif y\\
&\lesssim
\Lambda^{1/d}\beta.
\end{align*}
This proves \eqref{eq:198} and completes the proof of Proposition~\ref{prop:vdC-multidim}.
\end{proof}

\printbibliography

\end{document}